\documentclass{amsart}

\usepackage{paralist}
\usepackage{color}
\usepackage{mathrsfs}
\usepackage{amssymb}
\usepackage{amsmath}
\usepackage{amsfonts}
\usepackage{stmaryrd}

\newtheorem*{thma}{Theorem~A}
\newtheorem*{thmb}{Theorem~B}

\newtheorem{theorem}{Theorem}[section]
\newtheorem{lemma}[theorem]{Lemma}
\newtheorem{proposition}[theorem]{Proposition}
\newtheorem{corollary}[theorem]{Corollary}
\newtheorem{claim}[theorem]{Claim}
\newtheorem{subclaim}[theorem]{Subclaim}
\newtheorem{fact}[theorem]{Fact}
\newtheorem{question}[theorem]{Question}

\newtheorem*{q1}{Question~1}
\newtheorem*{q2}{Question~2}

\theoremstyle{definition}
\newtheorem{definition}[theorem]{Definition}
\newtheorem{remark}[theorem]{Remark}

\usepackage[pdftex]{hyperref}
\hypersetup{
    colorlinks=true, 
    linktoc=all,     
    linkcolor=blue,  
    allcolors=blue,
}

\usepackage[framemethod=tikz]{mdframed}

\newcommand{\cf}{\mathrm{cf}}
\newcommand{\dom}{\mathrm{dom}}
\newcommand{\supp}{\mathrm{supp}}
\newcommand{\bb}{\mathbb}

\newcommand{\otp}{\mathrm{otp}}

\newcommand{\lh}{\mathrm{lh}}

\newcommand{\mb}{\mathbf}
\newcommand{\mc}{\mathcal}

\newcommand{\sub}{\subseteq}
\newcommand{\ra}{\rightarrow}

\newcommand{\Et}{\bigwedge}

\newcommand{\ZFC}{\sf ZFC}

\newcommand{\CH}{\sf CH}

\newcommand{\Fn}{\mathrm{Fn}}
\newcommand{\wkd}{\mathrm{w}\diamondsuit}

\newcommand{\wl}{\mathrm{w} \lozenge}
\newcommand{\Lim}{\mathrm{Lim}}

\title{Simultaneously nonvanishing higher derived limits}

\author{Matteo Casarosa}
\address[Casarosa]{
Institut de Math\'ematiques de Jussieu - Paris Rive Gauche (IMJ-PRG)\\
Universit\'e Paris Cit\'e\\
B\^atiment Sophie Germain\\
8 Place Aur\'elie Nemours \\ 75013 Paris, France}
\address[Casarosa]{Dipartimento di Matematica, Universit\`{a} di Bologna, Piazza di
Porta S. Donato, 5, 40126 Bologna,\ Italy}
\address[Casarosa, current affiliation]{Departamento de Matem\'{a}ticas e Inform\'{a}tica, Universitat de Barcelona, Gran Via de les Corts Catalanes 585, 08007 Barcelona, Spain}
\email{matteo.casarosa@ub.edu}
\urladdr{https://sites.google.com/view/matteocasarosa}

\author{Chris Lambie-Hanson}
\address[Lambie-Hanson]{
Institute of Mathematics, 
Czech Academy of Sciences, 
{\v Z}itn{\'a} 25, 110 00 Prague 1, Czech Republic
}
\email{lambiehanson@math.cas.cz}
\urladdr{https://clambiehanson.github.io}

\thanks{The second author was supported by the Czech Academy of Sciences 
(RVO 67985840) and the GA\v{C}R project 23-04683S. Much of the research work took place during a visit of the first author supported by the Starting Grant 101077154 ``Definable Algebraic Topology" from the European Research Council, and partly by the GAČR project 23-04683S. The authors wish to thank Alessandro Vignati and Jeffrey Bergfalk for helpful conversations, and the anonymous referee for a number of helpful suggestions and corrections.}


\begin{document}

\begin{abstract}
The derived functors $\lim^n$ of the inverse limit find many applications in algebra and topology. In particular, the vanishing of certain derived limits $\lim^n \mb{A}[H]$, parametrized by an abelian group $H$, has implications for strong homology and condensed mathematics. In this paper, we prove that if $\mathfrak{d}=\omega_n$, then $\lim^n \mb{A}[H] \neq 0$ holds for $H=\bb{Z}^{(\omega_n)}$ (i.e. the direct sum of $\omega_n$-many copies of $\bb{Z}$). The same holds for $H=\bb{Z}$ under the additional assumption that $\wkd(S^{k+1}_k)$ holds for all $k < n$. In particular, this shows that if $\lim^n \mb{A}[H] = 0$ holds for all $n \geq 1$ and all abelian groups $H$, then $2^{\aleph_0} \geq \aleph_{\omega+1}$, thus answering a question of Bannister. Finally, we prove some consistency results regarding simultaneous nonvanishing of derived limits, again in the case of $H = \bb{Z}$. In particular, we show the consistency, relative to \ZFC, of $\bigwedge_{2 \leq k < \omega} \lim^k \mb{A} \neq 0$.
\end{abstract}

\maketitle

\section{Introduction}

The past few years have seen considerable progress in the application of set-theoretic 
tools to the study of homological algebra, and in particular to the study of the derived 
functors of the inverse limit functor. This paper is a contribution to this line of research. 
In particular, we isolate a number of situations in which certain derived limits are 
provably nonzero.

The inverse systems that we are primarily concerned with here are of the form 
$\mb{A}_{\mc{I}}[H]$, where $\mc{I}$ is a $\subseteq$-directed collection of sets 
and $H$ is an arbitrary abelian group. These systems are indexed along the 
directed partial order $(\mc{I}, \subseteq)$.\footnote{We refer the reader to  
Subsection \ref{subsec: derived_limits} for precise definitions.} We are especially 
interested in the case in which $\mc{I} = \{I_f \mid f \in {^{\omega}\omega}\}$, where 
\[
  I_f = \{(k,m) \in \omega \times \omega \mid m < f(k)\}.
\]
We will omit $\mc{I}$ from the notation $\mb{A}_{\mc{I}}[H]$ when $\mc{I}$ has this value; 
similarly, we will omit $H$ from the notation when $H = \bb{Z}$.

The derived limits of these systems, and questions about their vanishing, 
show up in a variety of contexts in various fields 
of mathematical research. To give two prominent examples:
\begin{itemize}
  \item In \cite{mp}, Marde\v{s}i\'{c} and Prasolov prove that, if strong homology is additive 
  on the class of all closed subsets of Euclidean space, then $\lim^n \mb{A} = 0$ for all 
  $1 \leq n < \omega$.
  \item Clausen and Scholze show that the assertion that $\lim^n \mb{A}[H] = 0$ for all 
  $1 \leq n < \omega$ and all abelian groups $H$ is equivalent to a useful statement about 
  the calculation of derived functors in the category of condensed abelian groups (for a precise 
  statement, see \cite[Lecture 4]{analytic_stacks} or \cite{blh_condensed}).
\end{itemize}

In the late 1980s and early 1990s, a number of works were published applying set-theoretic tools 
to the study of the \emph{first} derived limit of the system $\mb{A}$ (cf.\ \cite{mp, dsv, kamo,
todcmpct}, and see the introduction of \cite{svhdl} for a summary of the contents of these works). 
It was not until roughly ten years ago that the \emph{higher} derived limits of $\mb{A}$ and its 
relatives began to be explored. This was begun in \cite{bergfalk} and has continued in a number of 
works published since then. We now survey some of the relevant recent highlights of this research 
program, beginning with results about the consistent \emph{vanishing} of higher derived limits.
\begin{itemize}
  \item In \cite{svhdl}, Bergfalk and Lambie-Hanson prove that, relative to the consistency of 
  the existence of certain large cardinals, the statement ``for all $1 \leq n < \omega$, 
  $\lim^n \mb{A} = 0$" is consistent with $\ZFC$. In particular, they prove that, in any 
  forcing extension obtained by adding weakly compact-many Hechler reals, 
  $\lim^n \mb{A} = 0$ for all $1 \leq n < \omega$.
  \item In \cite{svhdlwlc}, Bergfalk, Hru\v{s}\'{a}k, and Lambie-Hanson remove the large 
  cardinal assumptions from the main result of \cite{svhdl}. In particular, they prove that, 
  in any forcing extension obtained by adding $\beth_\omega$-many Cohen reals, 
  $\lim^n \mb{A} = 0$ for all $1 \leq n < \omega$.
  \item In \cite{bannister}, Bannister sharpens the arguments from \cite{svhdl} and 
  \cite{svhdlwlc} to build models where derived limits vanish simultaneously for a broader class of systems. In particular, it follows from his work that, in the model of \cite{svhdlwlc}, we 
  in fact have $\lim^n \mb{A}[H] = 0$ for all $1 \leq n < \omega$ and all abelian groups 
  $H$.
\end{itemize}
In the other direction, there have been some recent results about the consistent 
\emph{nonvanishing} of the derived limits of $\mb{A}$.
\begin{itemize}
  \item In \cite{velickovic2021non}, Veli\v{c}kovi\'{c} and Vignati prove that, for each 
  $1 \leq n < \omega$, it is consistent with $\ZFC$ that $\lim^n \mb{A} \neq 0$. In particular, 
  they prove that, if $\mathfrak{b} = \mathfrak{d} = \aleph_n$ and $\wkd(S^{k+1}_k)$ holds for 
  all $k < n$, then $\lim^n \mb{A} \neq 0$.\footnote{$\mathfrak{b}$ and $\mathfrak{d}$ are the 
  \emph{bounding} and \emph{dominating} numbers, respectively. See definition \ref{def: wkd} 
  for the definition of $\wkd(S^{k+1}_k)$.}
  \item In \cite{casarosa2024nonvanishingderivedlimitsscales}, Casarosa shows that the 
  assumption of $\mathfrak{b} = \mathfrak{d}$ is not necessary in the main result from 
  \cite{velickovic2021non}. In particular, he proves that, for all $1 \leq k \leq n < \omega$, 
  it is consistent that $\mathfrak{b} = \aleph_k$, $\mathfrak{d} = \aleph_n$, and 
  $\lim^n \mb{A} \neq 0$.
\end{itemize}

A number of questions remained open in the wake of the work described above. Let us now recall 
two prominent such questions, which together comprise the primary motivation behind this paper. 

\begin{q1}
  What is the minimum value of $2^{\aleph_0}$ compatible with the statement ``for all $1 \leq n 
  < \omega$, $\lim^n \mb{A} = 0$"? What is the minimum value of $2^{\aleph_0}$ compatible with 
  the statement ``for all $1 \leq n < \omega$ and all abelian groups $H$, $\lim^n \mb{A}[H] = 0$"?
\end{q1}

The first half of Question 1 was asked in \cite{svhdl}, and also
in \cite{svhdlwlc} and \cite{bannister}. The second half of Question 1 is closely related to 
Bannister's \cite[Question 7.7]{bannister}, which asks an analogous question about a slighly broader class 
of inverse systems. The first half of the question remains open; we provide here an answer to 
the second half and, in the process, to \cite[Question 7.7]{bannister}. In particular, we 
prove the following theorem.

\begin{thma}
  Suppose that $1 \leq n < \omega$ and $\mathfrak{d} = \omega_n$. Then
  \begin{enumerate}
    \item $\lim^n \mb{A}[\bb{Z}^{(\omega_n)}] \neq 0$, where $\bb{Z}^{(\omega_n)}$ denotes 
    the direct sum of $\omega_n$-many copies of $\bb{Z}$;
    \item if, in addition, $\wkd(S^{k+1}_k)$ holds for all $k < n$, then $\lim^n \mb{A} \neq 0$.\footnote{As was recently shown by Bannister in \cite{bannister_nonvanishing}, 
    the assumption of $\wkd(S^1_0)$ is in fact superfluous here; it is enough to assume 
    $\wkd(S^{k+1}_k)$ for all $1 \leq k < n$.}
  \end{enumerate}
  In particular, if $\lim^n \mb{A}[H] = 0$ for all $1 \leq n < \omega$ and all abelian groups $H$, 
  then $2^{\aleph_0} \geq \aleph_{\omega+1}$.
\end{thma}

\cite[Question 7.7]{bannister} asks about the minimum value of the continuum compatible with 
additivity of derived limits for $\Omega_\omega$-systems. We refer the reader to 
\cite{bannister} for the relevant definitions; we just note here that a statement of the form 
$\lim^n \mb{A}[H] \neq 0$ provides a counterexample to the additivity of derived limits for 
$\Omega_\omega$-systems, and Bannister shows in \cite{bannister} that additivity of derived limits 
for $\Omega_\omega$-systems is compatible with $2^{\aleph_0} = \aleph_{\omega+1}$. Therefore, 
Theorem A provides an answer to \cite[Question 7.7]{bannister}. 

We now turn to the second question motivating this paper.

\begin{q2}
  Let $X$ be an arbitrary set of positive integers. Is there a model of $\ZFC$ in which 
  $\lim^n \mb{A} = 0$ if and only if $n \in X$? In particular, is the statement ``for all 
  $1 \leq n < \omega$, $\lim^n \mb{A} \neq 0$" consistent with $\ZFC$?
\end{q2}

Question 2 was asked in \cite{svhdlwlc}, and we make some partial progress towards it here. 
Notably, prior to the present paper, there was only one instance in which the simultaneous 
\emph{nonvanishing} of $\lim^n \mb{A}$ for multiple values of $1 \leq n < \omega$ was known to 
be consistent. Namely, it follows from results in \cite{todcmpct} and \cite{bergfalk} that it 
is consistent with $\ZFC$ that $\lim^1 \mb{A} \neq 0$ and $\lim^2 \mb{A} \neq 0$ simultaneously 
(in the model witnessing this, $\lim^n \mb{A} = 0$ for all $3 \leq n < \omega$). Here, we construct 
models witnessing the consistency of the simultaneous nonvanishing of $\lim^n \mb{A}$. In particular, 
we prove the following theorem.

\begin{thmb}
  \begin{enumerate}
    \item Fix $2 \leq n < \omega$. There exists a model of $\ZFC$ in which $\mathfrak{b} = 
    \mathfrak{d} = \aleph_n$ and $\bigwedge_{2 \leq k \leq n} \lim^k \mb{A} \neq 0$.
    \item There exists a model of $\ZFC$ in which $\bigwedge_{2 \leq k < \omega} \lim^k \mb{A} 
    \neq 0$.
  \end{enumerate}
\end{thmb}

Note that, by a theorem of Goblot \cite{goblot}, if $1 \leq n < \omega$ and  
$\mathfrak{d} = \aleph_n$, then $\lim^k \mb{A} = 0$ for all $k > n$. Thus, clause (1) 
of Theorem B is in one sense optimal. In our model for clause (2) of Theorem B, we 
have $\mathfrak{b} = \mathfrak{d} = \aleph_{\omega + 2}$. By the aforementioned result of 
Goblot, $\bigwedge_{2 \leq k < \omega} \lim^k \mb{A} \neq 0$ implies that 
$\mathfrak{d}$ is at least $\aleph_{\omega+1}$. Our result is thus almost optimal; it remains 
open whether one can obtain the same conclusion with $\mathfrak{d} = \aleph_{\omega+1}$. 
Also, the models that we construct in proving Theorem B will satisfy $\lim^1 \mb{A} = 0$. 
We therefore fall just short of answering the ``in particular" clause of Question 2, which 
remains open.

We now briefly discuss some of the methods underlying our results and the structure of 
the paper. As has become standard in this area of research, given $1 \leq n < \omega$, our 
verifications that a derived limit $\lim^n \mb{A}_{\mc{I}}[H]$ is nonzero in a particular 
model will reduce to the construction of a combinatorial object known as a 
\emph{nontrivial coherent $n$-family}. In Section \ref{sec: coherence}, we recall some 
of the basic definitions and facts surrounding nontrivial coherent $n$-families and 
their connections with derived limits.

Our proof of Theorem A builds upon the work of Veli\v{c}kovi\'{c} and Vignati in 
\cite{velickovic2021non}. In that paper, the assumption that $\mathfrak{b} = \mathfrak{d}$ 
yields the existence of a $\subseteq^*$-increasing cofinal sequence from 
$\{I_f \mid f \in {^\omega}\omega\}$ along which to perform recursive constructions of 
objects witnessing instances of $\lim^n \mb{A} \neq 0$. In the absence of the assumption 
that $\mathfrak{b} = \mathfrak{d}$, such sequences need not exist. As shown in \cite{casarosa2024nonvanishingderivedlimitsscales}, the weaker assumption of the existence of an unbounded chain is sufficient. To generalize these results even further, 
in Section \ref{sec: ascending} we introduce the notion of an \emph{ascending sequence} of sets. 
``Ascending" is a weakening of ``$\sub^*$-increasing", but we prove that it is strong enough to allow 
one to carry out modified versions of the constructions from \cite{velickovic2021non}.

In Section \ref{sec: ideals} we apply the technical results of Section \ref{sec: ascending} 
to obtain nonvanishing results for derived limits. In particular, we isolate natural conditions on 
an ideal $\mc{I}$ that imply the existence of an ascending $\sub^*$-cofinal sequence of elements 
of $\mc{I}$ to which one can apply the methods developed in Section \ref{sec: ascending}. 
We then apply these ideas to the specific ideal $\emptyset \times \mathrm{Fin}$, which is the 
ideal generated by the sets $\{I_f \mid f \in {^{\omega}}\omega\}$, to obtain a proof of 
Theorem A.

Finally, Section \ref{sec: squares} contains our proof of Theorem B. This proof again builds 
on the techniques of \cite{velickovic2021non}, combining them with the essential use 
of certain square sequences to enable the recursive construction of nontrivial coherent 
$n$-families of length greater than $\omega_n$. The heart of the section is a technical 
stepping-up lemma allowing us to use an appropriate combination of square principles and 
weak diamonds to construct nontrivial coherent $(n+1)$-families out of 
shorter nontrivial coherent $n$-families. This lemma is then applied in various forcing 
extensions to yield a proof of Theorem B.

\subsection{Notation and conventions}

Our notation is for the most part standard. For undefined notions in set theory, 
we refer the reader to \cite{jech}, and in homological algebra, to \cite{weibel}.

We identify an ordinal with the set of all ordinals strictly less than it. 
In particular, we identify the natural number $n$ with the set 
$\{0, 1, \ldots, n-1\}$. If $C$ is a set of ordinals, then we let $\Lim(C)$ denote 
the set $\{\alpha \in C \setminus \{0\} \mid \sup(C \cap \alpha) = \alpha\}$.
If $a$ and $b$ are sets of ordinals, then we let $a < b$ denote the assertion 
$(\forall \alpha \in a)(\forall \beta \in b)(\alpha < \beta)$. In particular, 
for any set $a$ of ordinals, we have $\emptyset < a$ and $a < \emptyset$, so this is 
only a partial order when restricted to nonempty sets.

If $\alpha$ is an ordinal, then $\cf(\alpha)$ denotes its cofinality. 
If $\lambda$ is a regular infinite cardinal and $\delta$ is an ordinal, then 
$S^\delta_\lambda = \{\alpha < \delta \mid \cf(\alpha) = \lambda\}$. If 
$m < n < \omega$, then we let $S^n_m$ denote $S^{\omega_n}_{\omega_m}$.

Given a positive integer $n$, we will identify 
functions with domain $n$ and sequences of length $n$, i.e., we will make 
no distinction between a function $u$ with domain $n$ and the sequence 
$\langle u(0), u(1), \ldots, u(n-1) \rangle$. Given such a $u$ and a natural 
number $i < n$, we let $u^i$ denote the sequence of length $n-1$ formed by 
removing $u(i)$ from $u$. This is often denoted 
\[
  \langle u(0), u(1), \ldots, \widehat{u(i)}, \ldots, u(n-1) \rangle;
\]
formally, it is the function $u^i$ with domain $n-1$ defined by 
\begin{align*}
  u^i(j) = \begin{cases}
    u(j) & \text{if } j < i \\
    u(j+1) & \text{if } j \geq i
  \end{cases}
\end{align*}
for all $j < n-1$. If $n < \omega$ and $\sigma \colon n \ra n$ is a permutation, then 
$\mathrm{sgn}(\sigma) \in \{-,+\}$ denotes the parity of $\sigma$.

Given a class $X$ and a cardinal $\kappa$, $[X]^\kappa$ denotes the 
class of all subsets of $X$ of cardinality $\kappa$.
Given a set of ordinals $a$, we let $\mathrm{otp}(a)$ denote its order-type.
We will customarily identify sets of ordinals with the functions 
enumerating them in increasing order, i.e., if $a \subseteq \mathrm{Ord}$ 
has order-type $\delta$ then, for all $i < \delta$, we let
$a(i)$ denotes the unique 
$\beta \in a$ such that $\otp(a \cap \beta) = i$. We will also regularly 
identify a sequence of length $1$ with its unique value, e.g., we will 
write $\varphi_\alpha$ instead of $\varphi_{\langle \alpha \rangle}$.
Similarly, we may write, e.g., $\varphi_{\alpha\beta}$ instead of 
$\varphi_{\langle \alpha, \beta \rangle}$. If $\sigma$ is a sequence 
indexed by an ordinal, then $\lh(\sigma)$ denotes its domain. If $\alpha < \beta$
are ordinals and $\sigma$ and $\tau$ are sequences of length $\alpha$ and 
$\beta$, respectively, then $\sigma \sqsubseteq \tau$ denotes the assertion 
that $\tau \restriction \alpha = \sigma$.

Given two functions $\varphi_0$ and $\varphi_1$, we write 
$\varphi_0 =^* \varphi_1$ to denote the assertion that, on their common 
domains, $\varphi_0$ and $\varphi_1$ agree at all but finitely many places. 
More formally, this is the assertion that the set 
\[
  \{x \in \dom(\varphi_0) \cap \dom(\varphi_1) \mid \varphi_0(x) \neq 
  \varphi_1(x)\}
\] 
is finite. As a special case, if $j$ is an element of the codomain of $\varphi$, 
then $\varphi =^* j$ is the assertion that $\{x \in \dom(\varphi_0) 
\mid \varphi(x) \neq j\}$ is finite. Similarly, if $u$ and $v$ are two sets, 
then we let $u \subseteq^* v$ denote the assertion that $u \setminus v$ is 
finite. If $\varphi$ is a function from a set $x$ into an abelian group $A$, 
then the \emph{support} of $\varphi$, denoted $\supp(\varphi)$, is 
$\{i \in x \mid \varphi(i) \neq 0\}$. If $f$ is a function and $X \subseteq \dom(f)$, 
then $f[X]$ denotes the pointwise image of $X$ under $f$. Similarly, if $X$ is a subset of the 
codomain of $f$, then $f^{-1}[X]$ denotes the preimage of $X$ under $f$. If 
$y$ is an element of the codomain of $f$, we will write $f^{-1}\{y\}$ instead of 
$f^{-1}[\{y\}]$.

We let $\mathsf{Ab}$ denote the category of abelian groups. 
We will be interested in functions mapping into abelian groups. For a set 
$u$ and an abelian group $H$, the set of functions from $u$ to $H$ itself has 
a natural abelian group structure defined by pointwise addition. For improved 
readability, we will slightly abuse notation in the following way: given sets 
$u_0$ and $u_1$, an abelian group $H$, and functions $\varphi_i \colon u_i \ra H$ 
for $i < 2$, we will write $\varphi_0 + \varphi_1$ instead of the more precise 
$\varphi_0 \restriction (u_0 \cap u_1) + \varphi_1 \restriction (u_0 \cap 
u_1)$. This generalizes straightforwardly to longer finite sums: if 
$n$ is a positive integer and $\langle \varphi_i \mid i < n \rangle$ is 
a family of functions mapping into $H$, then by convention the domain of the function 
$\sum_{i < n} \varphi_i$ is $\bigcap_{i < n} \dom(\varphi_i)$.

\section{Coherence, triviality, and derived limits} \label{sec: coherence}

In this section, we introduce the primary objects of study of this paper and review 
some of their basic properties.

\subsection{Nontrivial coherent families}

\begin{definition}
  Suppose that $\mc{I}$ is a collection of sets, $n$ is a positive integer, and 
  $H$ is an abelian group. An \emph{$H$-valued $n$-family indexed along 
  $\mc{I}$} is a family of functions of the form 
  \[
    \Phi = \left \langle \varphi_u \colon \bigcap_{i < n} u(i) \ra H ~ \middle| ~ u \in 
    \mc{I}^n \right \rangle.
  \]
  If $H$ and $\mc{I}$ are either clear from context or irrelevant, then we may refer 
  to such an object simply as an \emph{$n$-family}. An $n$-family is said to be
  \begin{itemize}
    \item \emph{alternating} if, for every $u \in \mc{I}^n$ and every permutation 
    $\sigma \colon n \ra n$, we have 
    \[
      \varphi_{u \circ \sigma} = \mathrm{sgn}(\sigma) \varphi_u;
    \]
    \item \emph{coherent} if it is alternating and, for all $v \in \mc{I}^{n+1}$, 
    we have 
    \[
      \sum_{i < n+1} (-1)^i \varphi_{v^i} =^* 0;
    \]
    \item \emph{trivial} if it is alternating and
    \begin{itemize}
      \item $n = 1$ and there is a function $\psi \colon \bigcup \mc I \ra H$ 
      such that, for all $u \in \mc{I}$, we have $\psi \restriction u 
      =^* \varphi_u$; or
      \item $n > 1$ and there is an alternating $(n-1)$-family 
      \[
        \Psi = \left\langle \psi_t \colon \bigcap_{i < n-1} t(i) \ra H \ \middle| \ t \in 
        \mc{I}^{n-1} \right\rangle
      \] 
      such that, for all $u \in \mc{I}^n$, we have
      \[
        \varphi_u =^* \sum_{i < n} (-1)^i \psi_{u^i}.
      \]
    \end{itemize}
    In this case, we say that $\psi$ or $\Psi$ \emph{trivializes} $\Phi$.
  \end{itemize}
\end{definition}

\begin{remark} \label{remark: coherent_trivial}
  It is easily verified that a trivial $n$-family is always coherent. 
  The question of when the converse holds is the primary motivating 
  question of this work.
  
  Moreover, if an $n$-family $\langle \varphi_u \mid u \in 
  \mc{I}^n \rangle$ is alternating, then, for every non-injective 
  $u \in \mc{I}^n$ we must have $\varphi_u = 0$. Therefore, when 
  constructing alternating $n$-families indexed along $\mc{I}$, it suffices 
  to specify $\varphi_u$ for injective $u \in \mc{I}^n$. 
  In practice, we will often be working with sets $\mc{I}$ equipped with 
  a natural linear order $<_{\mc{I}}$. In this setting, 
  we will let $[\mc{I}]^n$ denote the set of all $u \in \mc{I}^n$ such that 
  $u$ is strictly $<_{\mc{I}}$-increasing. Note then that,
  when constructing an alternating $n$-family indexed along $\mc{I}$, 
  it suffices to specify $\varphi_u$ for functions $u \in [\mc{I}]^n$, since 
  the requirement that the family be alternating will then determine all 
  other values. Similarly, when verifying that such a family is 
  either coherent or trivial, it suffices to consider 
  functions $v \in [\mc{I}]^{n+1}$ or $u \in [\mc{I}]^n$, respectively. When 
  working with an alternating $n$-family $\Phi = \langle \varphi_u \mid u \in 
  \mc{I}^n \rangle$ and a linear ordering $<_{\mc{I}}$ of $\mc{I}$, we will often 
  think of $\Phi$ as $\langle \varphi_u \mid u \in [\mc{I}]^n \rangle$, since this 
  is more economical and involves no loss of information. If $\mc{I}$ is enumerated 
  as $\langle x_\alpha \mid \alpha < \delta \rangle$, then, for ease of notation, when 
  given an alternating $n$-family $\Phi = \langle \varphi_u \colon
  \bigcap_{i < n} u(i) \ra H \mid u \in \mc{I}^n \rangle$, we will often rename it 
  as $\langle \varphi'_b \mid b \in [\delta]^n \rangle$ (or simply 
  $\langle \varphi_b \mid b \in [\delta]^n \rangle$ if there is no risk of confusion) where, 
  given $b \in [\delta]^n$, we set $\varphi'_b = \varphi_{\langle x_{b(i)} \mid 
  i < n \rangle}$. If we want to then show that, say, $\Phi$ is trivial, 
  in the case in which $n > 1$ it would then suffice to construct a family 
  $\langle \psi_a \mid a \in [\delta]^{n-1} \rangle$ such that, for all $b \in [\delta]^n$, 
  we have $\varphi_b =^* \sum_{i < n} (-1)^i \psi_{b^i}$.
  We will also sometimes discuss (alternating) $n$-families indexed along a sequence 
  $\vec{x} = \langle x_\alpha \mid \alpha < \delta \rangle$ of sets (note that there may 
  be $\alpha < \beta < \delta$ for which $x_\alpha = x_\beta$). In this setting, we
  mean families of the form
  \[
    \Phi = \langle \varphi_a \colon \bigcap_{i < n} x_{a(i)} \ra H \mid a \in \delta^n \rangle
    \hspace{0.2in} \text{ or } \hspace{0.2in} \Phi = \langle \varphi_a \colon \bigcap_{i < n} x_{a(i)} \ra H \mid 
    a \in [\delta]^n \rangle.
  \]
  The notions of \emph{coherence} and \emph{triviality} for such families are defined in the obvious ways.
\end{remark}

\begin{definition}
  Suppose that $\Phi = \langle \varphi_u \mid u \in \mc{I}^n \rangle$ is an 
  $n$-family. If $\mc{J} \subseteq \mc{I}$, then we let $\Phi \restriction 
  \mc{J}$ denote $\langle \varphi_u \mid u \in \mc{J}^n \rangle$. 
  
  If $\Phi = \langle \varphi_u \mid u \in \mc{I}^n \rangle$ and 
  $\Psi = \langle \psi_u \mid u \in \mc{J}^n \rangle$ are $n$-families, 
  then $\Psi \sqsubseteq \Phi$ denotes the assertion that $\Psi = \Phi \restriction 
  \mc{J}$ (we will sometimes say in this situation that $\Phi$ \emph{extends} 
  $\Psi$).
\end{definition}

Note that, if $\mc{J} \subseteq \mc{I}$ and $\Phi$ as above is 
coherent (\emph{resp.}\ trivial), then $\Phi \restriction \mc{J}$ is 
also coherent (\emph{resp.}\ trivial). The following is a partial converse to 
this; it is essentially a special case of 
analogous statements about derived limits appearing in work of Roos \cite{roos} 
and Jensen \cite{jensen_lim}, and later generalized to the context of cofinal 
functors by Mitchell \cite{mitchell_dimension}. We provide a proof for 
completeness.

\begin{proposition} \label{prop: cofinal_subset}
  Suppose that $\Phi = \langle \varphi_u \mid u \in \mc{I}^n \rangle$ is 
  a coherent $H$-valued $n$-family, $\mc{J}$ is a $\subseteq^*$-cofinal 
  subset of $\mc{I}$, and $\Phi \restriction \mc{J}$ is trivial. Then $\Phi$ is trivial.
\end{proposition}

\begin{proof}
  We assume that $n > 1$. The case in which $n = 1$ is easier and left to the reader.
  Let $\Psi = \langle \psi_t \mid t \in \mc{J}^{n-1} \rangle$ trivialize $\Phi 
  \restriction \mathcal{J}$. Fix an arbitrary linear order $<_{\mc{I}}$ of $\mc{I}$, 
  and recall the definition of $[\mc{I}]^m$ from Remark \ref{remark: coherent_trivial}. 
  By the same remark, in order to show that $\Phi$ is trivial, 
  it suffices to define an $(n-1)$-family 
  \[
    \left\langle \psi'_t \colon \bigcup_{i < n-1} t(i) \ra H ~ \middle| ~ 
    t \in [\mc{I}]^{n-1} \right\rangle
  \]
  such that, for all $u \in [\mc{I}]^n$, we have 
  $\varphi_u =^* \sum_{i<n} (-1)^i \psi'_{u^i}$.
  
  For each $x \in \mc{I}$, choose an $x^+ \in \mc{J}$ 
  such that $x \subseteq^* x^+$. Given $m < \omega$ and  $v \in \mc{I}^m$, define 
  \begin{itemize}
    \item $v^+ \in \mc{J}^m$ by letting $v^+(i) = v(i)^+$ for all $i < m$; and
    \item for $i < m$, define $\pi_i(v) \in \mc{I}^{m+1}$ as 
    \[
      \pi_i(v) = \langle v(0), \ldots, v(i), v(i)^+, \ldots, v(m-1)^+ \rangle.
    \]
  \end{itemize}
  Now define the $H$-valued $(n-1)$-family $\Psi' = \langle \psi'_t \mid t 
  \in [\mc{I}]^{n-1} \rangle$ by letting 
  \[
    \psi'_t = \psi_{t^+} - \sum_{i<n-1} (-1)^i \varphi_{\pi_i(t)}
  \]
  for all $t \in [\mc{I}]^{n-1}$ (the sum on the right is defined on all but finitely many 
  elements of $\bigcap_{i < n-1} t(i)$; $\psi'_t$ is defined more precisely by extending that 
  sum to the domain $\bigcap_{i < n-1} t(i)$ by setting it equal to $0$ on all otherwise undefined 
  arguments).
  
  We claim that $\Psi'$ trivializes $\Phi$. To verify this, fix $u \in [\mc{I}]^n$, 
  and consider 
  \begin{align} \label{triv_sum}
    \sum_{i < n} (-1)^i \psi'_{u^i} = \sum_{i < n} (-1)^i \left(\psi_{(u^i)^+} - 
    \sum_{j < n-1} (-1)^j \varphi_{\pi_j(u^i)}\right).
  \end{align}
  Note first that, since $\Psi$ trivializes $\Phi \restriction \mathcal{J}$, we have
  \[
    \sum_{i<n} (-1)^i \psi_{(u^i)^+} =^* \varphi_{u^+},
  \]
  and therefore the sum on the right-hand side of (\ref{triv_sum}) reduces to
  \begin{align} \label{triv_sum_2}
    \varphi_{u^+} + \sum_{i<n} \sum_{j<n-1} (-1)^{i+j+1} \varphi_{\pi_j(u^i)}.
  \end{align}
  For each $\ell < n$, let 
  \[
    A_\ell = \{(i,j) \in (n-1) \times (n-2) \mid \ell = j < i \text{ or } i \leq j = \ell -1\}.
  \]
  Note that $A_\ell$ consists precisely of those pairs $(i,j)$ for which the formal definition of 
  the sequence $\pi_j(u^i)$ contains $u_\ell$ and $(u_\ell)^+$ as consecutive elements. Now the sum 
  in (\ref{triv_sum_2}) becomes
  \[
    \varphi_{u^+} + \sum_{\ell < n} \sum_{(i,j) \in A_\ell} (-1)^{i+j+1} \varphi_{\pi_j(u^i)}.
  \]
  By the coherence of $\Phi$ applied to the sequence $\langle u(0), 
  u(0)^+, u(1)^+, \ldots, u(n-1)^+ \rangle$, we have
  \[
    \varphi_{u^+} + \sum_{(i,j) \in A_0} (-1)^{i+j+1} \varphi_{\pi_j(u^i)} =^* 
    \varphi_{\langle u(0), u(1)^+, \ldots, u(n-1)^+ \rangle}.
  \]
  More generally, for $\ell < n$, we have
  \begin{align*}
    \varphi_{\langle u(0), \ldots, u(\ell - 1), u(\ell)^+, \ldots, u(n-1)^+\rangle} 
    &+ \sum_{(i,j) \in A_\ell} (-1)^{i+j+1} \varphi_{\pi_j(u^i)} \\ &=^*
    \varphi_{\langle u(0), \ldots, u(\ell), u(\ell + 1)^+, \ldots, u(n-1)^+ \rangle}.
  \end{align*}
  By repeatedly applying these equalities, we obtain:
  \begin{align*}
    \sum_{i<n} (-1)^i \psi'_{\varphi_{u^i}} & = \varphi_{u^+} + \sum_{\ell < n} \sum_{(i,j) \in A_\ell} (-1)^{i+j+1} \varphi_{\pi_j(u^i)} \\ 
    &=^* \varphi_{\langle u(0), u(1)^+, \ldots, u(n-1)^+ \rangle} + 
    \sum_{1 \leq \ell < n} \sum_{(i,j) \in A_\ell} (-1)^{i+j+1} \varphi_{\pi_j(u^i)} \\ 
    &=^* \varphi_{\langle u(1),u(1),u(2)^+, \ldots, u(n-1)^+ \rangle} + 
    \sum_{2 \leq \ell < n} \sum_{(i,j) \in A_\ell} (-1)^{i+j+1} \varphi_{\pi_j(u^i)} \\ 
    & =^* \cdots \\ 
    & =^* \varphi_{\langle u(1), \ldots, u(n-2), u(n-1)^+ \rangle} + \sum_{(i,j) \in A_{n-1}} 
    (-1)^{i+j+1} \varphi_{\pi_j(u^i)} \\ 
    & =^* \varphi_u.
  \end{align*}
  Thus, $\Psi'$ does trivialize $\Phi$, as desired.
\end{proof}

We can also lift coherence from a $\subseteq^*$-cofinal subset of $\mc{I}$ to all of $\mc{I}$, 
in the following sense.

\begin{proposition} \label{prop: extension}
  Suppose that $\mc{J}$ is a $\subseteq^*$-cofinal subset of $\mc{I}$ and 
  $\Phi = \langle \varphi_u \mid u \in \mc{J}^n \rangle$ is a coherent $H$-valued $n$-family. 
  Then there is a coherent $H$-valued $n$-family $\Phi' = \langle \varphi'_u \mid u \in 
  \mc{I}^n \rangle$ such that $\Phi' \restriction \mc{J} = \Phi$.
\end{proposition}

\begin{proof}
  For each $x \in \mc{I}$, choose an $x^+ \in \mc{J}$ such that $x \subseteq^* x^+$. 
  If $x \in \mc{J}$, then choose $x^+ = x$. Given $m < \omega$ and $v \in \mc{I}^m$, 
  define $v^+ \in \mc{J}^m$ by letting $v^+(i) = v(i)^+$ for all $i < m$. 
  Now, for all $u \in \mc{I}^n$, define $\varphi'_u \colon \bigcap_{i < n} u(i) \ra H$ 
  by setting $\varphi'_u = \varphi_{u^+}$. Note that $\varphi_{u^+}$ is a function 
  defined on all but finitely many elements of $\bigcap_{i < n} u(i)$; $\varphi'_u$ is 
  defined more precisely by extending that sum to the domain $\bigcap_{i<n} u(i)$ 
  by setting it equal to $0$ on all otherwise undefined arguments.
  
  It is now immediate from the construction that $\Phi'$ is a coherent $n$-family 
  with $\Phi' \restriction \mc{J} = \Phi$.
\end{proof}

\begin{remark} \label{remark: extension}
  In light of Proposition  \ref{prop: extension}, 
  if we are given $0 < n < \omega$, an abelian group $H$, a collection $\mc{I}$ of sets, 
  and a $\subseteq^*$-cofinal subset $\mc{J} \subseteq \mc{I}$ and we seek to construct 
  a nontrivial coherent $H$-valued $n$-family indexed along $\mc{I}$, then it suffices to 
  construct such a family indexed along $\mc{J}$. Indeed, by Proposition \ref{prop: extension}, 
  this family extends to one indexed along $\mc{I}$, and the nontriviality of the original family 
  clearly implies the nontriviality of the extension. 
\end{remark}

The following is a variation on Goblot's vanishing theorem \cite{goblot} indicating that if 
the cofinality of $(\mc{I}, \subseteq^*)$ is less than $\omega_n$, then every coherent $n$-family 
indexed along $\mc{I}$ is trivial.\footnote{We note that the traditional statement 
of Goblot's theorem in the analogous setting requires the indexing set 
$\mc{I}$ to be $\subseteq$-directed, while we are not making that assumption 
in Proposition \ref{prop: goblot_var}. The reason we are able to forgo this 
assumption is that we are considering $n$-families of functions indexed by 
\emph{all} $n$-tuples from $\mc{I}$ rather than just those $n$-tuples that are 
$\subseteq$-increasing, which would correspond more directly to the classical 
setting of Goblot's theorem.}

\begin{proposition} \label{prop: goblot_var}
  Suppose that $H$ is an abelian group, $0 < n < \omega$, $\mc{I}$ is a collection of sets with
  $\cf(\mc{I}, \subseteq^*) < \omega_n$, and $\Phi = 
  \langle \varphi_u \mid u \in \mc{I}^n \rangle$ is an $H$-valued coherent 
  $n$-family. Then $\Phi$ is trivial.
\end{proposition}

\begin{proof}
  For concreteness, assume that $\cf(\mc{I}, \subseteq^*)$ is infinite; 
  the case in which it is finite is much easier. Let $\kappa = \cf(\mc{I}, 
  \subseteq^*)$. By replacing $\mc{I}$ with some $\subseteq^*$-cofinal 
  subset and invoking Proposition \ref{prop: cofinal_subset}, we can assume that 
  $|\mc{I}| = \kappa$. Enumerate $\mc{I}$ as $\langle x_\eta \mid \eta < 
  \kappa \rangle$. As noted in Remark \ref{remark: coherent_trivial}, 
  we will think of $\Phi$ as $\langle \varphi_b \mid b 
  \in [\kappa]^n \rangle$, where, formally, given 
  $b \in [\kappa]^n$, $\varphi_b = \varphi_{\langle x_{b(i)} \mid i < n \rangle}$.
  
  The proof is by induction on $n$. Suppose first that $n = 1$, and hence 
  $\kappa = \omega$. We will define a function $\psi \colon \bigcup \mc{I} \ra H$ 
  witnessing that $\Phi$ is trivial. For each $k < \omega$, let $y_k = 
  x_k \setminus (\bigcup \{x_j \mid j < k\})$; note that $\{y_k \mid k < \omega\}$ 
  is a partition of $\bigcup \mc{I}$. Now, for each $k < \omega$, set 
  $\psi \restriction y_k = \varphi_k \restriction y_k$.
  Using the coherence of $\Phi$, it is straightforward to verify that $\psi$ 
  thus defined witnesses that $\Phi$ is trivial.
  
  Now suppose that $n > 1$. We must construct an alternating $(n-1)$-family 
  \[
    \Psi = \langle \psi_t \mid t \in \mc{I}^{n-1} \rangle
  \]
  witnessing that $\Phi$ is trivial. Again recalling Remark \ref{remark: coherent_trivial}, 
  it suffices to construct a family
  \[
    \Psi = \langle \psi_a \mid a \in [\kappa]^{n-1} \rangle
  \]
  such that, for all $b \in [\kappa]^n$, we have
  \[
    \varphi_b =^* \sum_{i < n} (-1)^i \psi_{b^i}.
  \]
  
  The construction is slightly different depending on whether $n = 2$ or $n > 2$. 
  Suppose first that $n = 2$, in which case $\kappa \leq \omega_1$. We will 
  construct $\psi_\eta \colon x_\eta \ra H$ for $\eta < \kappa$ by recursion on $\eta$.
  
  Suppose that $\xi < \kappa$ and we have constructed $\langle \psi_\eta \mid 
  \eta < \xi \rangle$ such that, for all $\eta_0 < \eta_1 < \xi$, we have 
  $\varphi_{\eta_0\eta_1} =^* \psi_{\eta_1} - \psi_{\eta_0}$. Define a $1$-family of 
  functions
  \[
    T_\xi = \langle \tau_\eta \colon x_{\eta} \cap x_\xi \ra H \mid \eta < \xi \rangle
  \]
  by letting $\tau_\eta = \psi_\eta + \varphi_{\eta\xi}$.
  
  \begin{claim}
    $T_\xi$ is coherent.
  \end{claim}
  
  \begin{proof}
    Fix $\eta_0 < \eta_1 < \xi$. Then
    \begin{align*}
      \tau_{\eta_1} - \tau_{\eta_0} &= (\psi_{\eta_1} + \varphi_{\eta_1\xi}) 
      - (\psi_{\eta_0} + \varphi_{\eta_0\xi}) \\ &=^* \varphi_{\eta_1\xi} 
      - \varphi_{\eta_0\xi} + \varphi_{\eta_0\eta_1} \\ &=^* 0,
    \end{align*}
    where all functions are restricted to the domain $x_{\eta_0} \cap x_{\eta_1} 
    \cap x_\xi$, the passage from the first line to the second follows from 
    our assumptions about $\langle \psi_\eta \mid \eta < \xi \rangle$, and 
    the passage from the second line to the third follows from the 
    coherence of $\Phi$.
  \end{proof}
  
  Since $|\xi| < \omega_1$, the inductive hypothesis implies that $T_\xi$ is 
  trivial. We can therefore fix a function $\psi_\xi \colon x_\xi \ra H$ such 
  that, for all $\eta < \xi$, we have $\psi_\xi \restriction (x_\eta \cap x_\xi) 
  =^* \tau_\eta$. Note that, for all $\eta < \xi$, we have
  \[
    \psi_\xi - \psi_\eta =^* \tau_\eta - \psi_\eta = (\psi_\eta + \varphi_{\eta\xi}) 
    - \psi_\eta = \varphi_{\eta\xi},
  \]
  so this choice satisfies the requirements of the construction. This completes the 
  case $n = 2$.
  
  Suppose now that $n > 2$. We will construct $\psi_a \colon 
  \bigcap_{i < n-1} x_{a(i)} \ra H$ for $a \in [\kappa]^{n-1}$ by recursion on 
  $\max(a)$. 
  
  Suppose that $\xi < \kappa$ and we have constructed $\langle \psi_a \mid 
  a \in [\xi]^{n-1} \rangle$ such that, for all $b \in [\xi]^n$, we have
  \[
    \varphi_b =^* \sum_{i < n} (-1)^i \psi_{b^i}.
  \]
  We now describe how to define $\varphi_{d \cup \{\xi\}}$ for $d \in 
  [\xi]^{n-2}$. First, define an $(n-1)$-family of functions
  \[
    T_\xi = \langle \tau_a \colon x_\xi \cap \bigcap \{x_{a(i)} 
    \mid i < n-1\} \ra H \mid a \in [\xi]^{n-1} \rangle
  \]
  by letting $\tau_a = (-1)^n\psi_a + \varphi_{a \cup \{\xi\}}$.
  
  \begin{claim}
    $T_\xi$ is coherent.
  \end{claim}
  
  \begin{proof}
    Fix $b \in [\xi]^n$, and let $c = b \cup \{\xi\}$. Note that, 
    for $i < n-1$, we have $c^i = b^i \cup \{\xi\}$, and $c^{n-1} = b$. Then
    \begin{align*}
      \sum_{i < n} (-1)^i \tau_{b^i} &= (-1)^n \sum_{i < n} (-1)^i 
      \psi_{b^i} + \sum_{i < n} (-1)^i \varphi_{b^i \cup \{\xi\}} \\ 
      &=^* (-1)^n \varphi_b + \sum_{i < n} (-1)^i \varphi_{b^i \cup \{\xi\}} \\ 
      &= (-1)^n \varphi_{c^{n-1}} + \sum_{i < n} (-1)^i 
      \varphi_{c^i} \\ &= \sum_{i < n+1} (-1)^i \varphi_{c^i} =^* 0,
    \end{align*}
    where all functions are restricted to the domain $x_\xi \cap 
    \bigcap \{x_{b(i)} \mid i < n\}$, the passage from the first line to 
    the second follows from our assumptions about $\langle \psi_a \mid 
    a \in [\xi]^{n-1} \rangle$, the passage from the second line to the third 
    follows from the observations at the beginning of this proof, the 
    passage from the third line to the fourth is a simple rearranging of terms, 
    and the final equality (mod finite) follows from the coherence 
    of $\Phi$.
  \end{proof}
  
  Since $|\xi| < \kappa \leq \omega_{n-1}$, the inductive hypothesis implies 
  that $T_\xi$ is trivial. We can thus fix an $(n-2)$-family of functions
  \[
    \langle \psi_{d \cup \{\xi\}} \colon x_\xi \cap \bigcap \{x_{d(i)} \mid i < n-2\} \ra H
    \mid d \in [\xi]^{n-2} \rangle
  \]
  such that, for all $a \in [\xi]^{n-1}$, we have
  \[
    \tau_a =^* \sum_{i < n-1} (-1)^i \psi_{a^i \cup \{\xi\}}.
  \]
  We claim that this assignment of $\psi_{d \cup \{\xi\}}$ works. 
  To check this, fix an arbitrary $a \in [\xi]^{n-1}$, and let $b = a \cup \{\xi\}$. 
  Then
  \begin{align*}
    \sum_{i < n} (-1)^i \psi_{b^i} &= (-1)^{n-1} \psi_a + \sum_{i < n-1} 
    (-1)^i \psi_{a^i \cup \{\xi\}} \\ &=^* (-1)^{n-1} \psi_a + 
    \tau_a \\ &= (-1)^{n-1} \psi_a + (-1)^n \psi_a + \varphi_{a \cup \{\xi\}} 
    \\ &= \varphi_b.
  \end{align*}
  Thus, we can carry out the construction of a family $\langle \psi_a \mid a \in 
  [\kappa]^{n-1} \rangle$ witnessing that $\Phi$ is trivial, completing 
  the proof.
\end{proof}

\begin{definition}
  Suppose that $\vec{x} = \langle x_\alpha \mid \alpha < \delta \rangle$ 
  is a sequence of sets, $1 \leq n < \omega$, $H$ is an abelian group, and $\Phi = 
  \langle \varphi_a \colon \bigcap_{i < n} x_{a(i)} \ra H \mid u \in \delta^n 
  \rangle$ is an $n$-family of functions indexed along $\vec{x}$. If $E$ is a set, then $\Phi \restriction 
  \restriction E$ denotes the $n$-family 
  \[
    \left\langle \varphi_a \restriction E \cap \bigcap_{i < n} x_{a(i)} \ \middle| \ 
    a \in \delta^n \right\rangle
  \]
  indexed along $\vec{x} \restriction \restriction E := \langle x_\alpha \cap E \mid 
  \alpha < \delta \rangle$.
\end{definition}

Note that, if $\Phi$ is an alternating, coherent, or trivial $n$-family, then these 
properties are inherited by $\Phi \restriction \restriction E$ for every set $E$. 
In particular, if $\Phi$ is a coherent $n$-family and $\Phi \restriction \restriction E$ 
is nontrivial for some set $E$, then also $\Phi$ is nontrivial.

\begin{definition}
  Suppose that $1 < n < \omega$, $\delta$ is an ordinal, $H$ is an abelian group, 
  $\langle x_\alpha \mid \alpha < \delta \rangle$ is 
  a sequence of sets, and $\Phi = \langle \varphi_b \colon \bigcap_{i < n} x_{b(i)} \ra H
  \mid b \in \delta^n \rangle$ is a coherent $n$-family.
  Let $\mathrm{Triv}_H(\Phi)$ be the 
  set of all $(n-1)$-families $\Psi = \langle \psi_a \colon \bigcap_{i < n-1} x_{a(i)} \ra H 
  \mid a \in \delta^{n-1} \rangle$ that trivialize $\Phi$ 
  (so $\mathrm{Triv}_H(\Phi)$ is empty if $\Phi$ is nontrivial).
\end{definition}

If the group $H$ is clear from context, we may omit it from the notation 
$\mathrm{Triv}_H(\Phi)$.

\begin{proposition}
  Suppose that $1 < n < \omega$, $\delta$ is an ordinal, $H$ is an abelian group, 
  $\langle x_\alpha \mid \alpha < \delta \rangle$ is 
  a sequence of sets, and $\Phi = \langle \varphi_b \colon \bigcap_{i < n} x_{b(i)} \ra H
  \mid b \in \delta^n \rangle$ is a coherent $n$-family. 
  For all $\Psi^0, \Psi^1 \in \mathrm{Triv}_H(\Phi)$, the $(n-1)$-family
  \[
    \Psi^1 - \Psi^0 = \langle \psi^1_a - \psi^0_a \mid a \in \delta^{n-1} \rangle
  \]
  is coherent.
\end{proposition}

\begin{proof}
  Fix $\Psi^0, \Psi^1 \in \mathrm{Triv}_H(\Phi)$. The fact that $\Psi^1 - \Psi^0$ is 
  alternating follows immediately from the fact that both $\Psi^0$ and $\Psi^1$ are 
  alternating. To check coherence, fix $b \in \delta^n$. For each $\ell < 2$, the 
  fact that $\Psi^\ell \in \mathrm{Triv}_H(\Phi)$ implies that $\sum_{i < n} 
  \psi^\ell_{b^i} =^* \varphi_b$. Therefore, we have
  \[
    \sum_{i < n} (\psi^1_{b^i} - \psi^0_{b^i}) = \sum_{i < n} \psi^1_{b^i} - 
    \sum_{i < n} \psi^0_{b^i} =^* \varphi_b - \varphi_b = 0,
  \]
  as desired.
\end{proof}

\begin{definition}
  Given $1 < n < \omega$, an abelian group $H$, and an $H$-valued coherent 
  $n$-family $\Phi$, define an equivalence relation $\cong_{\Phi, H}$ on 
  $\mathrm{Triv}_H(\Phi)$ by setting $\Psi^0 \cong_{\Phi,H} \Psi^1$ if and only 
  if the coherent $(n-1)$-family $\Psi^1 - \Psi^0$ is trivial (via an $H$-valued 
  trivialization).
\end{definition}

The following basic facts are easily established; their proofs are left to 
the reader.

\begin{proposition} \label{basic_fact_prop}
  Suppose that $1 < n < \omega$, $H$ is an abelian group,
  $\Phi$ is an $H$-valued coherent $n$-family, and 
  $\Psi^0, \Psi^1 \in \mathrm{Triv}_H(\Phi)$.
  \begin{enumerate}
    \item For every set $E$, we have $\Psi^0 \restriction \restriction E, 
    \Psi^1 \restriction \restriction E 
    \in \mathrm{Triv}_H(\Phi \restriction \restriction E)$.
    \item If $E \subseteq F$ are sets and $\Psi^0 \restriction \restriction F 
    \cong_{\Phi \restriction \restriction F, H} \Psi^1 \restriction 
    \restriction F$, then $\Psi^0 \restriction \restriction E 
    \cong_{\Phi \restriction \restriction E, H} \Psi^1 \restriction 
    \restriction E$. \qed
  \end{enumerate}
\end{proposition}

\begin{proposition} \label{nonequiv_prop}
  Suppose that $1 < n < \omega$, $\delta$ is an ordinal, $H$ is an abelian group, 
  $\langle x_\alpha \mid \alpha < \delta \rangle$ is a sequence of sets, 
  and $\Phi = \langle \varphi_b \colon \bigcap_{i < n} x_{b(i)} 
  \ra H \mid b \in \delta^n \rangle$ is a coherent $n$-family. 
  Suppose also that $\Psi \in \mathrm{Triv}_H(\Phi)$ and $T = \langle \tau_a \colon 
  \bigcap_{i < n-1} x_{a(i)} \ra H \mid a \in \delta^{n-1} \rangle$ is a 
  coherent $(n-1)$-family. Then the family
  \[
    \Psi + T = \langle \psi_v + \tau_v \mid v \in \delta^{n-1} \rangle
  \]
  is in $\mathrm{Triv}_H(\Phi)$. Moreover, we have $\Psi + T \cong_{\Phi,H} \Psi$ 
  if and only if $T$ is trivial.
\end{proposition}

\begin{proof}
  We first verify that $\Psi + T$ is in $\mathrm{Triv}_H(\Phi)$. To this end, fix 
  $b \in \delta^n$. Then we have
  \begin{align*}
    \sum_{i<n} (\psi_{b^i} + \tau_{b^i}) & = \sum_{i < n} \psi_{b^i} + \sum_{i < n} \tau_{b^i} \\ 
    & =^* \varphi_b + 0 = \varphi_b,
  \end{align*}
  as desired. The ``moreover" clause follows immediately from the definition of 
  the equivalence relation $\cong_{\Phi,H}$.
\end{proof}

\begin{definition} \label{defn: ext_def}
  Suppose that 
  \begin{itemize}
    \item $1 < n < \omega$;
    \item $\delta$ is an ordinal;
    \item $\langle x_\alpha \mid \alpha < \delta + 1 \rangle$ is a sequence of sets;
    \item $H$ is an abelian group;
    \item $\Phi = \langle \varphi_b \colon \bigcap_{i < n} x_{b(i)} \ra H \mid 
    b \in \delta^n \rangle$ is a coherent $n$-family;
    \item $\Psi = \langle \psi_a \colon x_\delta \cap \bigcap_{i < n-1} 
    x_{a(i)} \ra H \mid a \in \delta^{n-1} \rangle \in 
    \mathrm{Triv}_H(\Phi \restriction \restriction x_\delta)$.
  \end{itemize}
  Then let $\Phi^\frown \langle \Psi \rangle$ denote the coherent $n$-family 
  $\langle \varphi_b \colon \bigcap_{i < n} x_{b(i)} \ra H \mid b \in (\delta+1)^n 
  \rangle$ extending $\Phi$ defined by letting $\varphi_{a^\frown \langle \delta 
  \rangle} = (-1)^{n+1} \psi_a$ for all $a \in \delta^{n-1}$.
\end{definition}

\begin{remark} \label{remark: coherence}
  Implicit in Definition \ref{defn: ext_def} is the assertion that the $n$-family 
  $\Phi^\frown \langle \Psi \rangle$ is indeed coherent. The verification of this 
  fact is routine and left to the reader.
\end{remark}

\begin{proposition} \label{nonextension_prop}
  Suppose that 
  \begin{itemize}
    \item $1 < n < \omega$;
    \item $\delta$ is an ordinal;
    \item $\langle x_\alpha \mid \alpha < \delta + 1 \rangle$ is a sequence of sets;
    \item $H$ is an abelian group;
    \item $\Phi = \langle \varphi_b \colon \bigcap_{i < n} x_{b(i)} \ra H \mid 
    b \in \delta^n \rangle$ is a coherent $n$-family;
    \item $\Psi \in \mathrm{Triv}_H(\Phi \restriction \restriction 
    x_\delta)$, $\Psi' \in \mathrm{Triv}_H(\Phi)$, and $\Psi  
    \not\cong_{\Phi \restriction \restriction x_\delta, 
    H} \Psi' \restriction \restriction x_\delta$.
  \end{itemize}
  Then $\Psi'$ does not extend to a trivialization of 
  $\Phi ^\frown \langle\Psi\rangle$.
\end{proposition}

\begin{proof}
  Assume that $n > 2$. The case $n = 2$ is similar but easier, and left to the reader.
  Let $\Phi ^\frown \langle \Psi \rangle = \langle \varphi_b \mid b \in (\delta+1)^n \rangle$.
  Suppose for the sake of contradiction that $\langle \psi'_a \mid a \in 
  (\delta+1)^{n-1} \rangle$ extends $\Psi'$ and trivializes $\Phi 
  ^\frown \langle\Psi\rangle$. Then, for all $a \in \delta^{n-1}$, we have
  \begin{align*}
    (-1)^{n+1} \psi_a & = \varphi_{a^\frown \langle \delta \rangle} \\ 
    & =^* \sum_{i < n} (-1)^i \psi'_{(a^\frown \langle \delta \rangle)^i} \\ 
    & = (-1)^{n-1} \psi'_{a} + \sum_{i < n-1} \psi'_{a^i{}^\frown \langle 
    \delta \rangle}.
  \end{align*}
  Rearranging the above equation, we see that, for all $a \in \delta^{n-1}$, 
  we have
  \[
    \psi_a - \psi'_a =^* (-1)^{n+1} \sum_{i < n-1} \psi'_{a^i{}^\frown 
    \langle \delta \rangle}.
  \]
  In particular, the family $\langle (-1)^{n+1} \psi'_{e^\frown \langle 
  \delta \rangle} \mid e \in \delta^{n-2} \rangle$ trivializes 
  $\Psi - \Psi' \restriction \restriction 
  x_\delta$, contradicting the fact that $\Psi 
  \not\cong_{\Phi \restriction \restriction x_\delta, 
  H} \Psi' \restriction \restriction x_\delta$.
\end{proof}

\subsection{Derived limits} \label{subsec: derived_limits}

One of the reasons for our interest in coherent $n$-families is their connection with 
the calculation of the derived limits of certain natural inverse systems of abelian groups.

\begin{definition}
  If $(\Lambda, \leq_\Lambda)$ is a directed partial order, then an \emph{inverse system 
  of abelian groups indexed by $\Lambda$} is a structure of the form
  \[
    \mb{B} = \langle B_u, \pi_{uv} \mid u,v \in \Lambda, \ u \leq_{\Lambda} v \rangle
  \]
  such that
  \begin{enumerate}
    \item for all $u \in \Lambda$, $B_u$ is an abelian group;
    \item for all $u \leq_\Lambda v$, $\pi_{uv} \colon B_v \ra B_u$ is a group homomorphism;
    \item for all $u \leq_\Lambda v \leq_\Lambda w$, $\pi_{uw} = \pi_{uv} \circ \pi_{vw}$.
  \end{enumerate}
  Given an inverse system $\mb{B}$, one can form the \emph{inverse limit} $\lim \mb{B} \in \mathsf{Ab}$. 
  Concretely, $\lim \mb{B}$ can be represented as 
  \[
    \{x \in \prod_{u \in \Lambda} B_u \mid \forall u \leq_\Lambda v, \ x(u) = \pi_{uv}(x(v))\}.
  \]
\end{definition}

Given a directed set $\Lambda$, one can form the category $\mathsf{Ab}^{\Lambda^{\mathrm{op}}}$ 
of all inverse systems of abelian groups indexed by $\Lambda$. (Since the details of this and what 
follows will not be necessary for us in this paper, we refer the reader to 
\cite[Section III]{mardesic_book} for precise definitions and proofs of facts presented in the 
remainder of this section.) The inverse limit operation then induces a functor 
$\lim\colon \mathsf{Ab}^{\Lambda^{\mathrm{op}}} \ra \mathsf{Ab}$. This functor is left exact but in 
general not exact; it therefore has (right) derived functors $\lim^n \colon 
\mathsf{Ab}^{\Lambda^{\mathrm{op}}} \ra \mathsf{Ab}$ for $1 \leq n < \omega$. The nontrivial coherent 
families of functions introduced earlier in this section play a key role in the computation 
of the derived limits of a particular class of inverse systems, which we now introduce.

\begin{definition}
  Suppose that $H$ is an abelian group.
  \begin{enumerate}
    \item Given a set $x$, let $A_x[H]$ denote the group $\bigoplus_x H$. Concretely, this 
    is the group of all finitely-supported functions from $x$ to $H$.
    \item Given sets $x \subseteq y$, let $\pi^H_{xy} \colon A_y[H] \ra A_x[H]$ be the natural restriction 
    map. More precisely, if $f \colon y \ra H$ is in $A_y[H]$, then $\pi^H_{xy}(f) = f \restriction x$. 
    Note that this is an abelian group homomorphism. If the group $H$ is clear from context, then 
    we will omit it from the notation.
    \item Suppose that $\mc{I}$ is a collection of sets that is $\subseteq$-directed, i.e., for all 
    $x,y \in \mc{I}$, there is $z \in \mc{I}$ such that $x \cup y \subseteq z$. Then 
    $\mb{A}_{\mc{I}}[H]$ denotes the inverse system of abelian groups 
    \[
      \langle A_x[H], \pi^H_{xy} \mid x,y \in \mc{I}, \ x \subseteq y \rangle.
    \]
  \end{enumerate}
\end{definition}

There is a particular $\subseteq$-directed collection of sets in which we will be especially
interested in this paper. Given a function $f\colon \omega \ra \omega$, let 
$I_f = \{(k,m) \in \omega \times \omega \mid m < f(k)\}$; i.e., $I_f$ is the region ``under the 
graph of $f$" in the plane $\omega \times \omega$. If we omit $\mc{I}$ from the notation 
$\mb{A}_{\mc{I}}[H]$, then $\mc{I}$ should be understood to be the set $\{I_f \mid 
f \in {^{\omega}}\omega\}$. Moreover, given $f \leq g \in {^{\omega}}\omega$, we will typically 
write $A_f[H]$ in place of $A_{I_f}[H]$ and $\pi_{fg}$ in place of $\pi_{I_fI_g}$, i.e., the inverse system 
$\mb{A}[H]$ is the system
\[
  \langle A_f[H], \pi_{fg} \mid f \leq g \in {^{\omega}}\omega \rangle.
\]
We will also omit mention of the group $H$ in case $H = \mathbb{Z}$, e.g., we will write 
$\mb{A}$ for $\mb{A}[\bb{Z}]$.

We can now state the connection between nontrivial coherent $n$-families and the derived limits 
of these inverse systems. For a proof of the fact, see \cite[Section III]{mardesic_book} or 
\cite[Section 2.2]{svhdl}.\footnote{The contents of \cite[Section 2.2]{svhdl} are specifically 
about the system $\mb{A}$, but the arguments therein straightforwardly generalize to 
$\mb{A}_{\mc{I}}[H]$ for any $\subseteq$-directed collection $\mc{I}$ and any abelian group $H$.}

\begin{fact} \label{fact: lim_coh}
  Suppose that $1 \leq n < \omega$, $\mc{I}$ is a $\subseteq$-directed collection of sets, and 
  $H$ is an abelian group. Then the following are equivalent:
  \begin{enumerate}
    \item $\lim^n \mb{A}_{\mc{I}}[H] = 0$;
    \item every coherent $H$-valued $n$-family indexed along $\mc{I}$ is trivial.
  \end{enumerate}
\end{fact}

Fact \ref{fact: lim_coh}, together with Propositions \ref{prop: cofinal_subset} and 
\ref{prop: extension} yield the following observation.

\begin{fact} \label{fact: cofinal_lim}
  Suppose that $1 \leq n < \omega$, $\mc{I}$ is a $\sub$-directed collection of sets, 
  $\mc{J}$ is a $\sub^*$-cofinal subset of $\mc{I}$, 
  and $H$ is an abelian group. Then the following are equivalent:
  \begin{enumerate}
    \item $\lim^n \mb{A}_{\mc{I}}[H] = 0$;
    \item every coherent $H$-valued $n$-family indexed along $\mc{J}$ is trivial.
  \end{enumerate}
\end{fact}

\section{Ascending sequences and nontrivial coherence} \label{sec: ascending}

In this section, we isolate the notion of an \emph{ascending sequence} of sets and prove 
that under certain assumptions one can recursively construct nontrivial coherent $n$-families 
along them.

For a nonempty set $Y$, we let $\Fn^+(Y)$ denote the collection of all finite partial 
functions $w$ from $Y$ to $2$ such that $w^{-1}\{1\} \neq \emptyset$.

\begin{definition}
  Suppose that $\vec{x} = \langle x_\alpha \mid \alpha < \delta \rangle$ 
  is a sequence of sets, $1 \leq n < \omega$, $H$ is an abelian group, and $\Phi = 
  \langle \varphi_a \colon \bigcap_{i < n} x_{a(i)} \ra H \mid a \in \delta^n 
  \rangle$ is an alternating $n$-family. 
  \begin{enumerate}
    \item If $E$ is a set, then we say 
    that $\Phi$ is \emph{supported on $E$} if $\supp(\varphi_a) \subseteq E$
    for all $a \in \delta^n$.
    \item Suppose that $\vec{e} = \langle e_w \mid w \in \Fn^+(\delta) \rangle$ 
    is a sequence of sets. For all $a \in \delta^n$, let 
    $\beta_a := \max\{a(i) \mid i < n\}$. We say that $\Phi$ is \emph{supported on 
    $\vec{e}$} if for all $a \in \delta^n$, we have $\supp(\varphi_a) 
    \subseteq \bigcup\{e_w \mid w \in \Fn^+(\beta_a)\}$.
  \end{enumerate}      
\end{definition}

The following proposition is immediate.

\begin{proposition}
  Suppose that $1 < n < \omega$, $H$ is an abelian group, $\Phi$ is an $H$-valued 
  coherent $n$-family, and $\Psi^0, \Psi^1 \in \mathrm{Triv}_H(\Phi)$. If $E$ is 
  a set, $\Psi^0$ and $\Psi^1$ are both supported on $E$, and $\Psi^0 \restriction 
  \restriction E \cong_{\Phi \restriction \restriction E, H} \Psi^1 \restriction 
  \restriction E$, then $\Psi^0 \cong_{\Phi, H} \Psi^1$. \qed
\end{proposition}

\begin{definition} \label{ascending_def}
  Suppose that $\delta$ is a limit ordinal of uncountable cofinality and 
  $\vec{x} = \langle x_\alpha \mid \alpha < \delta \rangle$ is a sequence 
  of sets.
  \begin{enumerate}
    \item For $w \in \Fn^+(\delta)$, we say that \emph{$\vec{x}$ respects 
    $w$} if the set
    \[
      d^{\vec{x}}_w := \bigcap_{\beta \in w^{-1}\{1\}} x_\beta
      \mathbin{\big\backslash} \bigcup_{\alpha \in w^{-1}\{0\}} x_\alpha
    \]
    is infinite.
    \item We say that $\vec{x}$ is \emph{ascending} if there is a sequence 
    $\langle e_w \mid w \in \Fn^+(\delta) \rangle$ and a club $C \subseteq 
    \delta$ satisfying the following requirements:
    \begin{enumerate}
      \item for all $w \in \Fn^+(\delta)$, if $\vec{x}$ respects $w$, then 
      $e_w \in [d^{\vec{x}}_w]^{\aleph_0}$ (otherwise $e_w = \emptyset$);
      \item \label{item_2b} for all $\beta \in C$ and all $a \in [\beta]^{<\omega}$, if $w \in \Fn^+(\delta)$ 
      is defined by setting 
      \[
        w(\alpha) = \begin{cases}
          0 & \text{if } \alpha \in a \\ 
          1 & \text{if } \alpha = \beta,
        \end{cases}
      \]
      then $\vec{x}$ respects $w$, i.e.,
      $x_\beta \not\subseteq^* \bigcup_{\alpha \in a} x_\alpha$;
      \item \label{item_2c} for all nonempty $a \in [C]^{<\omega}$ and all $w \in \Fn^+(\min(a))$, 
      if $\vec{x}$ respects $w$, then the set
      \[
        e_w \cap \bigcap_{\beta \in a} x_\beta
      \]
      is infinite. In particular, for all $\beta \in C$ and all $w \in \Fn^+(\beta)$, if 
      $\vec{x}$ respects $w$, then $\vec{x}$ respects $w \cup \{(\beta,1)\}$.
    \end{enumerate}
  \end{enumerate}
\end{definition}

\begin{remark} \label{remark_ascending_ex}
  Note that, as a consequence of clauses (\ref{item_2b}) and (\ref{item_2c}) of Definition \ref{ascending_def}, 
  if $\vec{x} = \langle x_\alpha \mid \alpha < \delta \rangle$ is ascending as witnessed by 
  a club $C \subseteq \delta$, then, for all $w \in \Fn^+(\delta)$ such that 
  $w^{-1}\{1\} \subseteq C$ and $w$ is \emph{weakly increasing} (i.e., $w^{-1}\{0\} 
  < w^{-1}\{1\}$), $\vec{x}$ respects $w$.
  
  To help illustrate Definition \ref{ascending_def}, we provide here, as suggested by the 
  referee, a couple of concrete examples of ascending sequences of sets; 
  we leave the routine verification that each of these examples is ascending as 
  an exercise for the reader. First, if $\vec{x}$ is any strictly $\subseteq^*$-increasing 
  sequence of sets, then $\vec{x}$ is ascending. Second, if 
  $\vec{x}$ enumerates any independent family of elements of $[\omega]^\omega$, then 
  $\vec{x}$ is ascending. Note that these examples occupy opposite extremes with respect to 
  the functions that are respected by the sequence: if $\vec{x} = 
  \langle x_\alpha \mid \alpha < \delta \rangle$ is a sequence of sets, then
  \begin{itemize}
    \item $\vec{x}$ is $\subseteq^*$-increasing if and only if $\vec{x}$ respects precisely 
    those elements of $\Fn^+(\delta)$ that are weakly increasing; and
    \item if $\vec{x}$ is a sequence from $[\omega]^\omega$, $\vec{x}$ enumerates an 
    independent family if and only if $\vec{x}$ respects \emph{all} 
    elements of $\Fn^+(\delta)$.
  \end{itemize}
\end{remark}

We now prove that, for $1 \leq n < \omega$, ascending sequences of length $\omega_n$ carry 
nontrivial coherent $n$-families. We first prove the existence of $n$-families mapping 
into arbitrary sufficiently small nonzero abelian groups under an additional weak 
diamond assumption. We will later eliminate this weak diamond assumption, at the cost of 
increasing the size of the groups that we are mapping into. Before beginning, let us 
recall the definition of weak diamond.

\begin{definition} \label{def: wkd}
Let $\lambda$ be a regular uncountable cardinal and let $S \sub \lambda $ be stationary in $\lambda$. The principle $\wl(S)$ asserts that for every $F\colon 2^{< \lambda} \to 2$ there exists $g\colon \lambda \to 2$ such that for every $b\colon \lambda \to 2$ the set 
\[ \lbrace \alpha \in S \mid g(\alpha) = F(b \restriction \alpha) \rbrace  \]
is stationary in $\lambda$.  
\end{definition}

We begin with the special case of $n = 1$.

\begin{theorem} \label{thm: 1d_ascending}
  Suppose that $\vec{x} = \langle x_\alpha \mid \alpha < \omega_1 \rangle$ 
  is an ascending sequence of sets, $H$ is a nonzero abelian group, and $\wkd(\omega_1)$ holds. 
  Then there is a nontrivial coherent 
  $1$-family $\Phi = \langle \varphi_\alpha \colon x_\alpha \ra H \mid \alpha 
  < \omega_1 \rangle$. 
  
  Moreover, if $\vec{e} = \langle e_w \mid w \in \Fn^+(\omega_1) 
  \rangle$ and $C$ witness that $\vec{x}$ is ascending, then we 
  can arrange so that $\Phi$ is supported on $\vec{e}$.
\end{theorem}

\begin{proof}
  Fix for the duration of the proof an arbitrary nonzero element of $H$, and denote it by $1_H$.
  We will denote the zero element of $H$ by $0_H$.
  For each $\sigma \in {^{<\omega_1}}2$, we will construct a coherent $H$-valued $1$-family 
  $\Phi^\sigma = \langle \varphi^\sigma_\alpha \mid \alpha < \lh(\sigma) \rangle$ 
  in such a way that if $\sigma \sqsubseteq \tau \in {^{<\omega_1}}2$, then 
  $\Phi^\sigma \sqsubseteq \Phi^\tau$.
  We will then find some $g \in {^{\omega_1}}2$ 
  for which the family $\Phi^g := \bigcup_{\beta < \omega_1} \Phi^{g \restriction 
  \beta}$ is nontrivial.
  
  Let $\vec{e} = \langle e_w \mid w \in \Fn^+(\omega_1) \rangle$ and $C \subseteq 
  \omega_1$ witness that $\vec{x}$ is ascending. Before we begin the 
  construction of $\langle \Phi^\sigma \mid \sigma \in {^{<\omega_1}}2 \rangle$, 
  we isolate, for each $\gamma \in \Lim(C)$, a relevant set 
  $y_\gamma \in [x_\gamma]^{\aleph_0}$.
  
  Fix $\gamma \in \Lim(C)$, as well as a bijection $\pi_\gamma \colon \gamma 
  \ra \omega$. Let $\langle \gamma_n \mid n < \omega \rangle$ be an 
  increasing cofinal sequence from $C \cap \gamma$ such that 
  $\langle \pi_\gamma(\gamma_n) \mid n < \omega \rangle$ is also increasing.
  For each $n < \omega$, define a function $w_{\gamma,n} \in \Fn^+(\gamma)$ by 
  letting 
  \[ 
    \dom(w_{\gamma,n}) := \{\gamma_n\} \cup \{\alpha < \gamma_n \mid 
    \pi_\gamma(\alpha) < \pi_\gamma(\gamma_n)\}
  \]
  and setting $w_{\gamma,n}(\gamma_n) := 1$ and $w_{\gamma,n}(\alpha) := 0$ for all 
  $\alpha \in \dom(w_n) \cap \gamma_n$. Then let $w_{\gamma,n}^+ := 
  w_{\gamma,n} \cup \{(\gamma,1)\}$. The fact that $\gamma_n < \gamma$ are both 
  in $C$ implies that $\vec{x}$ respects both $w_{\gamma,n}$ and 
  $w^+_{\gamma,n}$ and that $x_\gamma \cap e_{w_{\gamma,n}}$ is infinite.
  Note also that, for all $m < n < \omega$, we have $e_{w_{\gamma,m}} \cap 
  e_{w_{\gamma,n}} = \emptyset$, since $w_{\gamma,n}(\gamma_m) = 0$. 
  Now, for each $n < \omega$, choose a single element $i_{\gamma,n} \in 
  x_\gamma \cap e_{w_{\gamma,n}}$, and let 
  $y_\gamma := \{i_{\gamma,n} \mid n < \omega\}$. 
  Note that, by construction, the set $y_\gamma \cap x_\alpha$ is finite for 
  all $\alpha < \gamma$, since $w_{\gamma,n}(\alpha) = 0$ for all 
  sufficiently large $n < \omega$.
  
  We now turn to the construction of $\langle \Phi^\sigma \mid 
  \sigma \in {^{<\omega_1}}2 \rangle$. We will ensure throughout the 
  construction that each $\Phi^\sigma$ is supported on $\vec{e}$.
  If $\sigma \in {^{<\omega_1}}2$ and $\lh(\sigma)$ is a limit ordinal, then 
  we are obliged to set $\Phi^\sigma := \bigcup_{\beta < \lh(\sigma)} 
  \Phi^{\sigma \restriction \beta}$. Thus, to complete the construction, it suffices 
  to specify how to produce $\Phi^{\sigma^\frown \langle 0 \rangle}$ and 
  $\Phi^{\sigma^\frown \langle 1 \rangle}$ from $\Phi^\sigma$. To this 
  end, fix $\gamma < \omega_1$ and $\sigma \in {^\gamma}2$, and suppose 
  that we have constructed $\Phi^\sigma$. For readability, 
  we will write $\Phi^{\sigma,\ell}$ in place of $\Phi^{\sigma^{\frown} 
  \langle \ell \rangle}$ for $\ell < 2$.
  
  Fix $\ell < 2$. To construct $\Phi^{\sigma, \ell}$, it 
  is enough to define $\varphi^{\sigma, \ell}_\gamma$. 
  To begin, since $\gamma < \omega_1$, we can use Proposition \ref{prop: goblot_var} to 
  find a function $\psi \colon x_\gamma \ra H$ such that $\psi =^* 
  \varphi^\sigma_\alpha$ for all $\alpha < \gamma$. Moreover, since $\Phi^\sigma$ 
  is supported on $\vec{e}$, we can assume that the support of $\psi$ 
  is a subset of $\bigcup \{e_w \mid w \in \Fn^+(\gamma)\}$. If 
  $\gamma \notin \Lim(C)$, then simply let $\varphi^{\sigma, \ell}_\gamma = \psi$. If $\gamma \in \Lim(C)$, then define 
  $\varphi^{\sigma, \ell}_\gamma$ by setting
  \[
    \varphi^{\sigma, \ell}_\gamma(x) = 
    \begin{cases}
      \ell_H & \text{if } x \in y_\gamma \\
      \psi(x) & \text{if } x \in x_\gamma \setminus y_\gamma.
    \end{cases}
  \]
  Note that we still have $\varphi^{\sigma, \ell}_\gamma 
  =^* \varphi^\sigma_\alpha$ for all $\alpha < \gamma$, due to the fact 
  that $y_\gamma \cap x_\alpha$ is finite for all such $\alpha$. We have 
  also satisfied the requirement that each $\Phi^{\sigma, \ell}$ is 
  supported on $\vec{e}$. This completes the construction.
  
  We now prepare for an application of $\wkd(\omega_1)$. Let 
  $E := \bigcup\{e_w \mid w \in \Fn^+(\omega_1)\}$, and note that 
  $|E| \leq \aleph_1$ and $y_\gamma \subseteq E$ for all $\gamma 
  \in \Lim(C)$. Let $\langle j_\alpha \mid \alpha < |E| \rangle$ 
  be an injective enumeration of $E$. For $\gamma < 
  \omega_1$, let $E_\gamma := \{j_\alpha \mid \alpha < \gamma\}$. Let 
  $D$ be the set of $\gamma \in \Lim(C)$ such that, for all $w \in 
  \Fn^+(\gamma)$, we have $e_w \subseteq E_\gamma$, and note that $D$ is 
  club in $\omega_1$. Additionally note that, by construction, for each 
  $\gamma \in D$ we have $y_\gamma \subseteq E_\gamma$, and hence every 
  $\tau \in {^\gamma}2$ induces a function $h_\tau \colon y_\gamma \ra 2$ by 
  setting $h_\tau(j_\alpha) = \tau(\alpha)$ for all $\alpha < \gamma$ 
  such that $j_\alpha \in y_\gamma$.
  
  Now define a function $F\colon{^{<\omega_1}}2 \ra 2$ by setting $F(\tau) := 1$ 
  for all $\tau$ such that $\gamma := \lh(\tau) \in D$ and 
  $h_\tau =^* 0$, and setting $F(\tau) := 0$ for all other $\tau \in 
  {^{<\omega_1}}2$. Let $g \in {^{\omega_1}}2$ witness the instance of 
  $\wkd(\omega_1)$ associated with $F$. We claim that $\Phi^g = \bigcup_{\beta
  < \omega_1} \Phi^{g \restriction \beta} = 
  \langle \varphi_\alpha \mid \alpha < \omega_1 \rangle$ is nontrivial. 
  Suppose for the sake of contradiction that $\psi$ trivializes $\Phi^g$. 
  Define a function $f \in {^{\omega_1}}2$ by setting, for all $\alpha < \omega_1$,
  \[
    f(\alpha) = \begin{cases}
      0 & \text{if } \psi(j_\alpha) = 0_H \\ 
      1 & \text{otherwise.}
    \end{cases}
  \]  
  By our choice of $g$, 
  we can find $\gamma \in D$ such that $g(\gamma) = F(f \restriction \gamma)$.
  
  Let $\ell := g(\gamma)$ and $\tau := f \restriction \gamma$. By the construction 
  of $\Phi^g$, we know that $\varphi_\gamma \restriction y_\gamma =^* \ell_H$.
  By the assumption that $\psi$ trivializes $\Phi^g$, it follows 
  that $\psi \restriction y_\gamma =^* \ell_H$. Our definition of 
  $h_\tau$ then implies that $h_\tau =^* \ell$, but then our definition of 
  $F$ implies that $F(\tau) = 1-\ell$, contradicting the fact that 
  $g(\gamma) = F(f \restriction \gamma)$ and completing the proof. 
\end{proof}

We now turn to higher dimensions.

\begin{theorem} \label{thm: ascending}
  Suppose that $1 \leq n < \omega$, $\vec{x} = \langle x_\alpha \mid \alpha < \omega_n 
  \rangle$ is an ascending sequence of sets, $H$ is a nonzero abelian group with 
  $|H| \leq 2^{\omega_1}$, and $\wkd(S^{k+1}_k)$ holds for all $k < n$. 
  Then there is a nontrivial coherent $n$-family
  \[
    \Phi = \left\langle \varphi_b \colon \bigcap_{i < n} x_{b(i)} \rightarrow H ~ \middle| ~ b 
    \in (\omega_n)^n \right\rangle.
  \]
  Moreover, if $\vec{e} = \langle e_w \mid w \in \Fn^+(\omega_n) \rangle$ and 
  $C \subseteq \omega_n$ witness that $\vec{x}$ is ascending, then we can arrange 
  so that $\Phi$ is supported on $\vec{e}$. 
\end{theorem}

\begin{proof}
  The proof is by induction on $n$. The case of $n = 1$ is precisely Theorem \ref{thm: 1d_ascending}, 
  so assume that $n > 1$ and that we have established the theorem for all $1 \leq m < n$.
  Let $\langle e_w \mid w \in \Fn^+(\omega_n) \rangle$ and $C \subseteq \omega_n$ 
  witness that $\vec{x}$ is ascending.
  For $\gamma < \omega_n$, let $E_\gamma := \bigcup\{e_w \mid w \in 
  \Fn^+(\gamma)\}$, and let $\bar{E}_\gamma := E_\gamma \cap x_\gamma$.
  
  Let $\mc{S} := \{\sigma \in {^{<\omega_n}}2 \mid \supp(\sigma) \subseteq \Lim(C) \cap S^n_{n-1}\}$, 
  and let $\mc{S}^+ := \{g \in {^{\omega_n}}2 \mid \supp(g) \subseteq \Lim(C) \cap S^n_{n-1}\}$.
  Similarly to the proof of Theorem \ref{thm: 1d_ascending}, for each $\sigma \in \mc{S}$, we will 
  construct a coherent $H$-valued $n$-family $\Phi^\sigma = \langle \varphi^\sigma_b \mid b 
  \in \lh(\sigma)^n \rangle$ in such a way that if $\sigma \sqsubseteq \tau \in \mc{S}$, then $\Phi^\sigma 
  \sqsubseteq \Phi^\tau$. We will then find some $g \in \mc{S}^+$ such that 
  $\Phi^g := \bigcup_{\alpha < \omega_n} \Phi^{g \restriction \alpha}$ is nontrivial.
  
  We now describe the construction of $\langle \Phi^\sigma \mid \sigma \in 
  \mc{S} \rangle$, which will be done by recursion on $\lh(\sigma)$. 
  We will ensure along the way that, for all $\gamma < \omega_n$ and all 
  $\sigma \in \mc{S} \cap {^\gamma}2$, the family $\Phi^\sigma$ is supported on 
  $\langle e_w \mid w \in \Fn^+(\gamma) \rangle$.
  As in the proof of Theorem \ref{thm: 1d_ascending}, the case in which $\lh(\sigma)$ 
  is a limit ordinal is trivial, so fix a $\sigma \in \mc{S}$ and assume 
  that $\Phi^\sigma$ has been constructed. We describe how to construct 
  $\Phi^{\sigma, \ell}$ for $\ell < 2$ if $\lh(\sigma) \in 
  \Lim(C) \cap S^n_{n-1}$ and for $\ell = 0$ otherwise.
  
  Let $\gamma := \lh(\sigma)$. Since $\gamma < \omega_n$, Proposition
  \ref{prop: goblot_var} implies that $\Phi^\sigma \restriction \restriction x_\gamma$ 
  is trivial. Let $\Psi^\sigma = \langle \psi^\sigma_a \mid a \in \gamma^{n-1} 
  \rangle$ be an element of $\mathrm{Triv}_H(\Phi^\sigma \restriction 
  \restriction x_\gamma)$. Since $\Phi^\sigma \restriction \restriction x_\gamma$ 
  is supported on $\bar{E}_\gamma$, we can assume that $\Psi^\sigma$ is also 
  supported on $\bar{E}_\gamma$. If $\gamma \notin \Lim(C) \cap S^n_{n-1}$, 
  then simply let $\Phi^{\sigma, 0} = 
  \Phi^\sigma{}^\frown \langle \Psi^\sigma \rangle$. It is easily verified that this 
  maintains the requirements of the construction; to see that it maintains coherence, 
  recall Remark \ref{remark: coherence}.
  
  If $\gamma \in \Lim(C) \cap S^n_{n-1}$, then we work a little bit harder. 
  Let $D_\gamma \subseteq C \cap \gamma$ be a club in $\gamma$ with 
  $\otp(D_\gamma) = \omega_{n-1}$, and let $\pi_\gamma \colon \gamma \ra 
  \omega_{n-1}$ be a bijection such that $\pi \restriction 
  D_\gamma$ is order-preserving and continuous. In particular, 
  $\bar{D}_\gamma := \pi[D_\gamma]$ is a club in $\omega_{n-1}$. 
  Note that $\{\eta \in \bar{D}_\gamma \mid \pi_\gamma^{-1}[\eta] \subseteq 
  \pi^{-1}(\eta)\}$ is a club in $\omega_{n-1}$. By thinning out $D_\gamma$ 
  (and hence also $\bar{D}_\gamma)$, we can assume that $\pi_\gamma^{-1}[\eta] 
  \subseteq \pi_\gamma^{-1}(\eta)$ for all $\eta \in \bar{D}_\gamma$.
  
  For all $\eta < \omega_{n-1}$, let $y^\gamma_\eta := x_\gamma \cap 
  x_{\pi_\gamma^{-1}(\eta)}$. Given $w \in \Fn^+(\omega_{n-1})$, let 
  $\hat{w}^\gamma \in \Fn^+(\gamma)$ be defined by 
  letting $\dom(\hat{w}^\gamma) := \pi_\gamma^{-1}[\dom(w)]$ and, for all 
  $\alpha \in \dom(\hat{w}^\gamma)$, setting $\hat{w}^\gamma(\alpha) =
  w(\pi_\gamma(\alpha))$. For such $w$, let $\bar{e}^\gamma_w := 
  x_\gamma \cap e_{\hat{w}^\gamma}$.
  
  \begin{claim}
    $\vec{y} = \langle y^\gamma_\eta \mid \eta < \omega_{n-1} \rangle$ 
    is an ascending sequence of sets, as witnessed by $\langle \bar{e}^\gamma_w 
    \mid w \in \Fn^+(\omega_{n-1}) \rangle$ and $\bar{D}_\gamma$.
  \end{claim}
  
  \begin{proof}
    We must verify conditions (2)(a--c) of Definition \ref{ascending_def}. To see 
    (a), fix $w \in \Fn^+(\omega_{n-1})$, and note that 
    \[
      d^{\vec{y}}_w = d^{\vec{x}}_{\hat{w}^\gamma} \cap x_\gamma = 
      d^{\vec{x}}_{\hat{w}^\gamma \cup \{(\gamma,1)\}}.
    \] 
    By the fact that $\gamma \in C$, we know that
    \[
      (\vec{y} \text{ respects } w) \Leftrightarrow (\vec{x} \text{ respects } 
      \hat{w}^\gamma \cup \{(\gamma,1\}) 
      \Leftrightarrow (\vec{x} \text{ respects } \hat{w}^\gamma).
    \]
    Thus, if $\vec{y}$ respects $w$, the fact that $\vec{x}$ is ascending 
    implies that 
    \[
      \bar{e}^\gamma_w = x_\gamma \cap e_{\hat{w}^\gamma} 
      \in [d^{\vec{y}}_w]^{\aleph_0},
    \] 
    so (a) is satisfied.
    
    Item (b) follows from the same reasoning: fix $\eta \in \bar{D}_\gamma$ and 
    $a \in [\eta]^{<\omega}$, and define $w \in \Fn^+(\omega_{n-1})$ by setting 
    $\dom(w) = a \cup \{\eta\}$, $w(\eta) = 1$, and $w(\xi) = 0$ for all $\xi \in a$. 
    The fact that $\eta \in \bar{D}_\gamma$ and $\vec{x}$ is ascending implies that 
    $\vec{x}$ respects $\hat{w}^\gamma$, and hence $\vec{y}$ respects $w$, as desired.
    
    For (c), fix a nonempty $a \in [\bar{D}_\gamma]^{<\omega}$ and $w \in \Fn^+(\min(a))$ 
    such that $\vec{y}$ respects $w$. Let $\hat{a}^\gamma := \pi_\gamma^{-1}[a]$. 
    Note that $\hat{w}^\gamma \in \Fn^+(\min(\hat{a}^\gamma))$, 
    $\hat{a}^\gamma \in [C]^{<\omega}$, and $\vec{x}$ respects $\hat{w}^\gamma$.
    Therefore, since $\gamma \in C$, we know that $e_{\hat{w}^\gamma} 
    \cap x_\gamma \cap \bigcap_{\alpha \in \hat{a}^\gamma} x_\alpha$ is infinite. 
    But $e_{\hat{w}^\gamma} \cap x_\gamma \cap \bigcap_{\alpha \in \hat{a}^\gamma}
    x_\alpha = \bar{e}^\gamma_w \cap \bigcap_{\eta \in a} y_\eta$, so this 
    instance of (c) is satisfied.
  \end{proof}
  By the inductive hypothesis, we can find a nontrivial coherent $(n-1)$-family 
  \[
    \bar{T}^\gamma = \left\langle \bar{\tau}^\gamma_a \colon \bigcap_{i < n-1} y_{a(i)} \ra H \ 
    \middle| \ a \in (\omega_{n-1})^{n-1} \right\rangle
  \] 
  that is supported on $\bar{E}_\gamma$.
  Through re-indexing via $\pi_\gamma$, this yields a nontrivial coherent $(n-1)$-family
  \[
    T^\gamma = \left\langle \tau^\gamma_a \colon x_\gamma \cap \bigcap_{i < n-1} 
    x_{a(i)} \ra H \ \middle| \ a \in \gamma^{n-1} \right\rangle
  \]
  that is also supported on $\bar{E}_\gamma$. Then Proposition \ref{nonequiv_prop} 
  implies that $(\Psi^\sigma + T^\gamma) \in \mathrm{Triv}_H(\Phi^\sigma 
  \restriction \restriction x_\gamma)$ and $\Psi^\sigma 
  \not\cong_{\Phi^\sigma \restriction \restriction x_\gamma,H} (\Psi^\sigma + T^\gamma)$.
  Finally, set $\Phi^{\sigma, 0} := \Phi^\sigma{}^\frown 
  \langle \Psi^\sigma \rangle$, and set $\Phi^{\sigma, 1} 
  := \Phi^\sigma{}^\frown \langle \Psi^\sigma + T^\gamma \rangle$. It is routine to verify 
  that these definitions maintain coherence 
  and that the resulting families $\Phi^{\sigma, \ell}$ 
  are supported on $\langle e_w \mid w \in \Fn^+(\gamma+1) \rangle$. 
  This completes the construction of $\langle \Phi^\sigma \mid \sigma \in 
  \mc{S} \rangle$.
  
  Let $E = E_{\omega_n} := \bigcup\{e_w \mid w \in \Fn^+(\omega_n)\}$. 
  Note that the sequence $\langle E_\gamma \mid \gamma \leq \omega_n 
  \rangle$ is $\subseteq$-increasing and continuous, and that 
  $|E_\gamma| < \omega_n$ for all $\gamma < \omega_n$. 
  
  For all $\gamma \leq \omega_n$, let 
  $\mc P^\gamma$ denote the set of all alternating $(n-1)$-families of the form 
  \[
    \Psi = \left\langle \psi_a \colon E_\gamma \cap \bigcap_{i < n-1} x_{a(i)} \ra H 
    \ \middle| \ a \in \gamma^{n-1} \right\rangle.
  \]
  If $\gamma < \delta \leq \omega_n$, $\Psi \in \mc P^\gamma$, and $\Psi' \in 
  \mc P^\delta$, then we say that $\Psi'$ \emph{extends} $\Psi$, written 
  $\Psi' \trianglerighteq \Psi$, if, for all $a \in \gamma^{n-1}$, we have 
  $\psi_a = \psi'_a \restriction (E_\gamma \cap \bigcap_{i < n-1} x_{a(i)})$.
  
  \begin{claim} \label{claim: unique_triv}
    For all $\gamma < \omega_n$ and all $\Psi \in \mc{P}^\gamma$, there is at most 
    one $\sigma \in \mc{S} \cap {^\gamma}2$ such that $\Psi \in 
    \mathrm{Triv}_H(\Phi^\sigma \restriction \restriction E_\gamma)$.
  \end{claim}
  
  \begin{proof}
    Suppose for the sake of contradiction that there are $\gamma < \omega_n$, 
    $\Psi \in \mc{P}_\gamma$, and distinct $\sigma_0, \sigma_1 \in \mc{S} \cap {^{\gamma}}2$ 
    such that $\Psi \in \mathrm{Triv}_H(\Phi^{\sigma_\ell} \restriction \restriction 
    E_\gamma)$ for $\ell < 2$. Let $\beta < \gamma$ be the least ordinal at which 
    $\sigma_0$ and $\sigma_1$ disagree; without loss of generality, assume that 
    $\sigma_\ell(\beta) = \ell$ for $\ell < 2$. By the definition of $\mc{S}$, we must 
    have $\beta \in \Lim(C) \cap S^n_{n-1}$. Let $\tau = \sigma_0 \restriction \beta = 
    \sigma_1 \restriction \beta$. In the construction, we fixed a 
    $\Psi^\tau \in \mathrm{Triv}_H(\Phi^\tau \restriction \restriction x_\beta)$ and 
    a nontrivial coherent $(n-1)$-family $T^\beta$, both supported on $\bar{E}_\beta$, 
    such that $\Phi^{\tau, 0} = \Phi^\tau{}^\frown \langle 
    \Psi^\tau \rangle$ and $\Phi^{\tau, 1} = \Phi^\tau{}^\frown 
    \langle \Psi^\tau + T^\beta \rangle$. Then, noting that $\bar{E}_\beta \subseteq 
    E_\gamma$, our assumptions imply that, for all $a \in \beta^{n-1}$, we have
	\begin{align*}
	  (-1)^{n+1}(\psi^\tau_a) \restriction \bar{E}_\beta &= 
	  \varphi^{\sigma_0}_{a^\frown \langle \beta \rangle} \restriction \bar{E}_\beta \\ 
	  &=^* (-1)^{n-1} \psi_a \restriction \bar{E}_\beta + \sum_{i < n-1} (-1)^i\psi_{a^i{}^\frown 
	  \langle \beta \rangle} \restriction \bar{E}_\beta \\
	  &=^* \varphi^{\sigma_1}_{a^\frown \langle \beta \rangle} \restriction \bar{E}_\beta \\ 
	  &= (-1)^{n+1}(\psi^\tau_a + \tau^\beta_a) \restriction \bar{E}_\beta.
	\end{align*}
	In particular, from the first and last terms of the above equation we see that 
	$\tau^\beta_a \restriction \bar{E}_\beta =^* 0$ for all $a \in \beta^{n-1}$, 
	contradicting the fact that $T^\beta$ is a nontrivial coherent $(n-1)$-family 
	supported on $\bar{E}_\beta$.
  \end{proof}
  
  Fix a coding function $G$ with domain ${^{\leq \omega_n}}2$ and a club 
  $D \subseteq \Lim(C)$ with the following properties:
  \begin{itemize}
    \item for all $\gamma \in D \cup \{\omega_n\}$, $G \restriction {^{\gamma}}2$ is a surjection 
    from ${^{\gamma}}2$ to $\mc P^\gamma$ (this is where we use the assumption that 
    $|H| \leq 2^{\omega_1}$);
    \item for all $\gamma < \delta$, both in $D \cup \{\omega_n\}$, and all $\sigma \in {^{\gamma}}2$ 
    and $\sigma' \in {^{\delta}}2$, if $\sigma' \sqsupseteq \sigma$, then 
    $G(\sigma') \trianglerighteq G(\sigma)$.
  \end{itemize}
  As long as all intervals between successive elements of $D$ have size $\omega_{n-1}$, 
  it is routine to construct such a function $G$, using the fact that $|\mc P^\gamma| \leq 
  2^{\omega_{n-1}}$ for each $\gamma < \omega_n$.
  
  Now define a function $F\colon{^{<\omega_n}}2 \ra 2$ as follows. If 
  $\gamma \notin D \cap S^n_{n-1}$, simply let $F \restriction {^\gamma}2 = 0$. 
  Suppose now that $\gamma \in D \cap S^n_{n-1}$ and $\rho \in {^\gamma}2$.
  If there does not exist $\sigma \in \mc{S} \cap {^\gamma}2$ such that 
  $G(\rho) \in \mathrm{Triv}_H(\Phi^\sigma \restriction \restriction E_\gamma)$, 
  then let $F(\rho) = 0$. If there is $\sigma \in \mc{S} \cap {^\gamma}2$ 
  such that $G(\rho) \in \mathrm{Triv}_H(\Phi^\sigma \restriction \restriction E_\gamma)$, 
  then, by Claim \ref{claim: unique_triv}, there is a unique such $\sigma$; denote 
  this by $\sigma(\rho)$. By Proposition \ref{nonextension_prop} and the construction 
  of $\Phi^{\sigma(\rho), \ell}$, there is at most one $\ell < 2$ 
  such that $G(\rho)$ extends to an element of $\mathrm{Triv}_H(\Phi^{\sigma(\rho), \ell})$. 
  More precisely, if $G(\rho) = \Psi$, then there is at most one 
  $\ell < 2$ for which there exists a $\Psi' \in \mathrm{Triv}_H(\Phi^{\sigma(\rho), \ell})$ 
  such that, for all $a \in \gamma^{n-1}$, we have $\psi'_a \restriction E_\gamma = \psi_a$. 
  Then choose $F(\rho) \in 2$ so that $G(\rho)$ does \emph{not} extend to an element of 
  $\mathrm{Triv}_H(\Phi^{\sigma(\rho),\ell})$.
  
  Let $g \in {^{\omega_n}}2$ witness the instance of $\wkd(S^n_{n-1})$ associated with $F$. 
  Since this only depends on the values $g$ takes on $S^n_{n-1}$, or any relatively club 
  subset thereof, we can assume that $g \in \mc{S}$.
  We claim that $\Phi^g$ is nontrivial. Suppose for the sake of contradiction that $\Psi = 
  \langle \psi_a \mid a \in (\omega_n)^{n-1} \rangle$ trivializes 
  $\Phi^g$. By the construction of $G$, we can find a function $f \in 
  {^{\omega_n}}2$ such that, for all $\gamma \in D \cup \{\omega_n\}$, 
  we have 
  \[
    G(f \restriction \gamma) = \left\langle \psi_a \restriction 
    E_\gamma \cap \bigcap_{i < n-1} x_{a(i)} \ \middle| \ a \in 
    \gamma^{n-1} \right\rangle.
  \]
  Moreover, for all $\gamma \in D$, it follows that $G(f \restriction 
  \gamma)$ trivializes $\Phi^{g \restriction \gamma} \restriction 
  \restriction E_\gamma$.
  
  By our choice of $g$, we can find $\gamma \in D \cap S^n_{n-1}$ 
  such that $g(\gamma) = F(f \restriction \gamma)$. By our construction 
  of $F$, we then know that $\sigma(f \restriction \gamma) = g \restriction \gamma$
  and that $G(f \restriction \gamma)$ does not extend to an element of 
  $\mathrm{Triv}_H(\Phi^{g \restriction (\gamma + 1)})$. Therefore, 
  \emph{a fortiori}, $G(f \restriction \gamma)$ does not extend to an element 
  of $\mathrm{Triv}_H(\Phi^g)$. But this is contradicted by the fact 
  that $\Psi$ extends $G(f \restriction \gamma)$ and is an element of 
  $\mathrm{Triv}_H(\Phi^g)$. This contradiction shows that 
  $\Phi^g$ is indeed nontrivial, thus completing the proof of the theorem.
\end{proof}

\begin{remark}
  The question of whether the assumption that $|H| \leq 2^{\omega_1}$ can be 
  removed from the statement of Theorem \ref{thm: ascending} is a good one, and one 
  asked by the referee. The assumption could clearly be removed if one could prove 
  the following statement: for all $1 \leq n < \omega$ and all nonzero abelian groups 
  $G \leq H$, every nontrivial coherent $G$-valued $n$-family remains nontrivial 
  when viewed as an $H$-valued family, i.e., $(\mathrm{Triv}_G(\Phi) = \emptyset) 
  \Rightarrow (\mathrm{Triv}_H(\Phi) = \emptyset)$. This statement is clearly true when 
  $n = 1$ and, with some work, can be proven when $n = 2$. It is also true for all $n$ if 
  $G$ is a direct summand of $H$. However, whether the statement holds in full generality 
  remains open, as does the question of whether the cardinality assumption can be 
  removed from Theorem \ref{thm: ascending}.
\end{remark}

We can get rid of the weak diamond assumptions in the previous theorem at the 
cost of increasing the size of the group into which we are mapping. 
For a group $H$ and an ordinal $\beta$, we write $H^{(\beta)}$ 
to denote the direct sum of $\beta$-many copies of $H$; concretely, this is 
the group consisting of all finitely-supported functions $f \colon \beta \ra H$.
If $\alpha < \beta$, then we consider $H^{(\alpha)}$ as a subgroup of 
$H^{(\beta)}$ in the obvious way. If we have fixed a particular nonzero 
element $1_H \in H$, then, for $\eta < \beta$, we let 
$1_\eta$ denote the element of $H^{(\beta)}$ whose support is precisely $\{\eta\}$
and takes value $1_H$ at $\eta$.

\begin{theorem} \label{thm: ascending_wide}
  Suppose that $1 \leq n < \omega$, $H$ is a nonzero abelian group, 
  and $\vec{x} = \langle x_\alpha \mid \alpha < 
  \omega_n \rangle$ is an ascending sequence of sets. Then there is a nontrivial 
  coherent $n$-family
  \[
    \Phi = \left \langle \varphi_b \colon \bigcap_{i < n} x_{b(i)} \ra H^{(\omega_n)} 
    \ \middle| \ b \in (\omega_n)^n \right \rangle.
  \]
  Moreover, if $\vec{e} = \langle e_w \mid w \in \Fn^+(\omega_n) \rangle$ and 
  $C \subseteq \omega_n$ witness that $\vec{x}$ is ascending, then we can arrange 
  so that $\Phi$ is supported on $\vec{e}$.
\end{theorem}

\begin{proof}
  Fix for the remainder of the proof an arbitrary nonzero element of $H$, and denote it by $1_H$.
  The proof is by induction on $n$. Suppose first that $n = 1$, and fix 
  $\vec{e}$ and $C$ witnessing that $\vec{x}$ is ascending. We will construct 
  a nontrivial coherent 1-family $\Phi = \langle \varphi_\beta \colon x_\beta \ra 
  H^{(\omega_1)} \mid \beta < \omega_1 \rangle$ that is supported on $\vec{e}$  
  by recursion on $\beta$, maintaining 
  the recursive hypothesis that, for all $\beta < \omega_1$, $\varphi_\beta$ 
  maps into $H^{(\beta + 1)}$ (i.e., for all $z \in x_\beta$, the support of 
  $\varphi_\beta(z)$ is a subset of $\beta + 1$).
  
  Suppose that $\beta < \omega_1$ and we have constructed $\Phi^\beta = \langle \varphi_\alpha 
  \mid \alpha < \beta \rangle$. As before, let $E_\beta = \bigcup \{e_w \mid 
  w \in \Fn^+(\beta)\}$ and $\bar{E}_\beta = E_\beta \cap x_\beta$. Since $\beta < \omega_1$, 
  Proposition \ref{prop: goblot_var} implies that $\Phi^\beta \restriction \restriction x_\beta$ 
  is trivial. Fix a trivialization $\psi^\beta \colon x_\beta \ra H^{(\omega_1)}$. By 
  our recursive assumption, $\Phi^\beta \restriction \restriction x_\beta$ is supported 
  on $\bar{E}_\beta$ and, for each $\alpha < \beta$, $\varphi_\alpha$ maps into $H^{(\beta)}$.
  Therefore, we can assume that $\psi^\beta$ is also supported on $\bar{E}_\beta$ 
  and maps into $H^{(\beta)}$. If $\beta \notin \Lim(C)$, then simply let 
  $\varphi_\beta = \psi^\beta$; this readily maintains the recursive hypotheses.
  
  Suppose now that $\beta \in \Lim(C)$. Precisely as in the proof of Theorem \ref{thm: 1d_ascending}, 
  fix a set $y_\beta \in [\bar{E}_\beta]^\omega$ such that, for all $\alpha < \beta$, 
  $y_\beta \cap x_\alpha$ is finite.
%
%
%
Now define $\varphi_\beta$ by setting, for all $z \in x_\beta$,
  \begin{align*}
    \varphi_\beta(z) = \begin{cases}
      1_\beta & \text{if } z \in y_\beta \\
      \psi^\beta(z) & \text{if } z \notin y_\beta.
    \end{cases}  
  \end{align*}
  Since $y_\beta \cap x_\alpha$ is finite for all $\alpha < \beta$, this maintains the 
  coherence of $\Phi$. It is easily verified that all other recursive requirements are 
  satisfied by this definition. This completes the construction of $\Phi$.
  
  We claim that $\Phi$ is nontrivial. Suppose for the sake of contradiction that 
  $\psi \colon \bigcup \{x_\beta \mid \beta < \omega_1\} \ra H^{(\omega_1)}$ trivializes 
  $\Phi$. Note that $\langle E_\beta \mid \beta < \omega_1 \rangle$ is a continuous 
  $\subseteq$-increasing sequence of countable sets. Therefore, since, for each 
  $z \in \dom(\psi)$, $\supp(\psi(z))$ is a finite subset of $\omega_1$, we can find 
  $\beta \in \Lim(C)$ such that $\psi \restriction E_\beta$ maps into $H^{(\beta)}$.
  However, when we constructed $\varphi_\beta$, we found an infinite subset 
  $y_\beta \subseteq \bar{E}_\beta$ and ensured that $\beta \in \supp(\varphi_\beta(z))$ 
  for all $z \in y_\beta$. This implies that $\psi \restriction x_\beta \neq^* 
  \varphi_\beta$, contradicting the assumption that $\psi$ trivializes $\Phi$.
  
  Suppose now that $1 < n < \omega$ and we have established the theorem for all $1 \leq m < n$. 
  Again fix $\vec{e}$ and $C$ witnessing that $\vec{x}$ is ascending.
  For notational convenience, we will index our construction of a nontrivial coherent $n$-family 
  by $[\omega_n]^n$ rather than the full $(\omega_n)^n$; recall that this involves no loss of 
  generality by Remark \ref{remark: coherent_trivial}.
  We will construct a nontrivial coherent $n$-family $\Phi = \langle \varphi_b \colon 
  \bigcap_{i < n} x_{b(i)} \ra H^{(\omega_n)} \mid b \in [\omega_n]^n \rangle$ by recursion on $\max(b)$. 
  We will ensure that $\Phi$ is supported on $\vec{e}$ and will maintain the 
  recursive hypothesis that, for all $b \in [\omega_n]^n$, $\varphi_b$ maps into 
  $H^{(\max(b) + \omega_{n-1})}$.
  
  Fix $\gamma < \omega_n$ and suppose that we have specified 
  $\Phi^\gamma = \langle \varphi_b \mid b \in [\gamma]^n \rangle$. 
  We describe how to specify $\varphi_{a \cup \{\gamma\}}$ 
  for $a \in [\gamma]^{n-1}$. As before, let $E_\gamma := \bigcup\{e_w \mid 
  w \in \Fn^+(\gamma)\}$, and let $\bar{E}_\gamma := E_\gamma \cap x_\gamma$. 
  Since $\cf(\gamma) < \omega_n$, Proposition \ref{prop: goblot_var} implies that 
  $\Phi^\gamma \restriction \restriction x_\gamma$ is trivial. Fix 
  $\Psi^\gamma = \langle \psi^\gamma_a \colon x_\gamma \cap \bigcap_{i < n-1} x_{a(i)} 
  \ra H^{(\omega_{n+1})} \mid a \in [\gamma]^{n-1} \rangle$ in $\mathrm{Triv}_{H^{(\omega_n)}} 
  (\Phi^\gamma \restriction \restriction x_\gamma)$. Since $\Phi^\gamma \restriction \restriction 
  x_\gamma$ is supported on $\bar{E}_\gamma$ and, for each $b \in [\gamma]^n$, $\varphi_b$
  maps into $H^{(\gamma + \omega_{n-1})}$, 
  we can assume that $\Psi^\gamma$ is also supported on $\bar{E}_\gamma$ and, for 
  each $a \in [\gamma]^{n-1}$, $\psi^\gamma_a$
  maps into $H^{(\gamma + \omega_{n-1})}$. Moreover, if $\cf(\gamma) = \omega_{n-1}$, then we 
  in fact know that, for each $b \in [\gamma]^n$, $\varphi_b$ maps into 
  $H^{(\gamma)}$, so we can require that $\psi_a^\gamma$ does for each $a \in [\gamma]^{n-1}$ as well.
  
  If $\gamma \notin \Lim(C) \cap S^n_{n-1}$, then simply let $\varphi_{a \cup \{\gamma\}} 
  = (-1)^{n+1} \psi^\gamma_a$ for all $a \in [\gamma]^{n-1}$. Suppose now that $\gamma \in 
  \Lim(C) \cap S^n_{n-1}$. Recall in this case that $\psi_a^\gamma$ maps into 
  $2^{(\gamma)}$ for each $a \in [\gamma]^{n-1}$. 
  First, exactly as in the proof of Theorem \ref{thm: ascending}, 
  the inductive hypothesis implies that we can find a nontrivial coherent 
  $(n-1)$-family
  \[
    T^\gamma = \left\langle \tau^\gamma_a \colon x_\gamma \cap \bigcap_{i < n-1} 
    x_{a(i)} \ra H^{(\omega_{n-1})} \ \middle| \ a \in [\gamma]^{n-1} \right\rangle
  \]
  that is supported on $\bar{E}_\gamma$. The idea now is to add a ``shifted" version 
  of $T^\gamma$ to $\Psi^\gamma$. More precisely, for each $a \in [\gamma]^{n-1}$, 
  define a function 
  \[
    \mathrm{sh}_\gamma(\tau^\gamma_a)\colon x_\gamma \cap \bigcap_{i < n-1} x_{a(i)} 
    \ra H^{(\gamma + \omega_{n-1}))}
  \]
  by setting, for all $z \in x_\gamma \cap \bigcap_{i < n-1} x_{a(i)}$ and 
  all $\eta < \gamma + \omega_{n-1}$,
  \begin{align*}
    \mathrm{sh}_\gamma(\tau^\gamma_a)(z)(\eta) = 
    \begin{cases}
      0 & \text{if } \eta < \gamma \\ 
      \tau^\gamma_a(z)(\xi) & \text{if } \eta = \gamma + \xi.
    \end{cases}
  \end{align*}
  Then let $\varphi_{a \cup \{\gamma\}} = (-1)^{n+1}(\psi^\gamma_a + \mathrm{sh}_\gamma
  (\tau^\gamma_a)))$ for all $a \in [\gamma]^{n-1}$. It is easily verified, using the coherence 
  of $T^\gamma$ and the fact that $\Psi^\gamma$ trivializes $\Phi^\gamma \restriction \restriction 
  x_\gamma$, that this maintains coherence and all of the 
  other requirements of the recursive construction.
  
  We claim that the family $\Phi$ thus constructed is nontrivial. For the sake of contradiction, 
  suppose that it is trivial, and fix a $\Psi = \langle \psi_a \mid a \in 
  [\omega_n]^{n-1} \rangle$ in $\mathrm{Triv}_{H^{(\omega_n)}}(\Phi)$. 
  Using the fact that $\langle E_\gamma \mid \gamma < \omega_n \rangle$ is a continuous, 
  $\subseteq$-increasing sequence of sets, each of cardinality at most $\omega_{n-1}$,
  find $\gamma \in \Lim(C) \cap S^n_{n-1}$ such that, for all $a \in [\gamma]^n$, 
  $\psi_a \restriction E_\gamma$ maps into $H^{(\gamma)}$. 
  
  The argument now differs slightly depending on whether $n = 2$ or $n > 2$. Suppose 
  first that $n = 2$. Define a function $\rho \colon \bar{E}_\gamma \ra H^{(\omega_{n-1})}$ 
  by setting, for all $z \in \bar{E}_\gamma$ and all $\xi < \omega_{n-1}$,
  $\rho(z)(\xi) = -\psi_{\gamma}(z)(\gamma + \xi)$. By our choice of $\gamma$ and 
  the construction of $\Phi$, we know that, for all $\alpha < \gamma$, we have
  \begin{itemize}
    \item $\psi_{\gamma} - \psi_{\alpha} =^* \varphi_{\alpha\gamma}$;
    \item $\psi_\alpha \restriction \bar{E}_\gamma$ maps into $H^{(\gamma)}$;
    \item for all $z \in x_\alpha \cap \bar{E}_\gamma$ and all $\xi < \omega_{n-1}$, 
    $\varphi_{\alpha\gamma}(z)(\gamma + \xi) = -\tau^\gamma_\alpha(z)(\xi)$.
  \end{itemize}
  Putting this all together, we see that $\rho \restriction (x_\alpha \cap \bar{E}_\gamma) =^* 
  \tau^\gamma_\alpha \restriction \bar{E}_\gamma$ for all $\alpha < \gamma$, 
  and hence $\rho$ witnesses that $T^\gamma \restriction \restriction \bar{E}_\gamma$ is 
  trivial, contradicting the fact that $T^\gamma$ is nontrivial and supported on 
  $\bar{E}_\gamma$.
  
  Suppose next that $n > 2$. For each $d \in [\gamma]^{n-2}$, define a function 
  $\rho_d \colon \bar{E}_\gamma \cap \bigcap_{i < n-2} x_{d(i)} \ra H^{(\omega_{n-1})}$ by setting, 
  for all $z \in \bar{E}_\gamma \cap \bigcap_{i < n-2} x_{d(i)}$ and all $\xi < \omega_{n-1}$, 
  $\rho_d(z)(\xi) = (-1)^{n+1}\psi_{d \cup \{\gamma\}}(z)(\gamma + \xi)$. By our choice 
  of $\gamma$ and the construction of $\Phi$, we know that, for all $a \in [\gamma]^{n-1}$, we have
  \begin{itemize}
    \item $(-1)^{n-1} \psi_a + \sum_{i < n-1} (-1)^i \psi_{a^i \cup \{\gamma\}} 
    =^* \varphi_{a \cup \{\gamma\}}$;
    \item $\psi_a \restriction \bar{E}_\gamma$ maps into $H^{(\gamma)}$;
    \item for all $z \in \bar{E}_\gamma \cap \bigcap_{i < n-1} x_{a(i)}$ and all 
    $\xi < \omega_{n-1}$, 
    \[
      \varphi_{a \cup \{\gamma\}}(z)(\gamma + \xi) = (-1)^{n+1} \tau^\gamma_a(z)(\xi).
    \]
  \end{itemize}
  Putting this together, we see that, for all $a \in [\gamma]^{n-1}$, we have
  \[
    \sum_{i < n-1} (-1)^i \rho_{a^i} =^* \tau^\gamma_a 
    \restriction \bar{E}_\gamma,
  \]
  and hence $\langle \rho_d \mid d \in [\gamma]^{n-2} \rangle$ witnesses that 
  $T^\gamma \restriction \restriction \bar{E}_\gamma$ is trivial, contradicting 
  the fact that $T^\gamma$ is nontrivial and supported on $\bar{E}_\gamma$ 
  and completing the proof of the theorem.
\end{proof}

Theorems \ref{thm: ascending} and \ref{thm: ascending_wide}, together with Remark 
\ref{remark_ascending_ex}, immediately yield the following concrete corollary.

\begin{corollary} \label{cor: concrete}
  Suppose that $1 \leq n < \omega$, $H$ is a nonzero abelian group, 
  and $\vec{x} = \langle x_\alpha \mid \alpha < \omega_n 
  \rangle$ is a sequence of sets such that either
  \begin{itemize}
    \item $\vec{x}$ is strictly $\subseteq^*$-increasing; or
    \item $\vec{x}$ enumerates an independent family in $[\omega]^\omega$.
  \end{itemize}
  Then there exists a nontrivial coherent $H^{(\omega_n)}$-valued $n$-family indexed 
  along $\vec{x}$. If, moreover, $|H| \leq 2^{\omega_1}$ and 
  $\wkd(S^{k+1}_k)$ holds for all $k < n$, then there exists a nontrivial coherent 
  $H$-valued $n$-family indexed along $\vec{x}$.
\end{corollary}

\section{Ideals and their characteristics} \label{sec: ideals}

In this section, we isolate some properties of ideals $\mc{I}$ that ensure 
that they carry $\subseteq^*$-cofinal ascending sequences; this will then yield 
a proof of Theorem A.

The following cardinal characteristics associated with ideals were introduced in 
\cite{hernandez_hrusak} in the context of tall ideals on $\omega$. 
We reproduce the definition here in a more general context.

\begin{definition}
  Suppose that $\mc{I}$ is an ideal on a set $X$ properly extending the ideal of 
  all finite subsets of $X$ with no $\sub^*$-maximal set.
  \begin{itemize}
    \item $\mathrm{cof}^*(\mc{I}) = \min\{|\mc{Y}| \mid \mc{Y} \subseteq \mc{I} \wedge 
    (\forall x \in \mc{I})(\exists y \in \mc{Y})(x \subseteq^* y)\}$.
    \item $\mathrm{non}^*(\mc{I}) = \min\{|\mc{E}| \mid \mc{E} \subseteq \mc{I} \cap [X]^\omega 
    \wedge (\forall x \in \mc{I})(\exists e \in \mc{E})(|e \cap x| < \aleph_0)\}$.\footnote{In 
    \cite{hernandez_hrusak} the set $\mc{E}$ is only assumed to be a subset of $[X]^\omega$, 
    not $\mc{I} \cap [X]^\omega$. In the context of tall ideals on $\omega$ this does 
    not change the definition; in the general context, the definition we give here seems the 
    more useful one, at least for our purposes.}
  \end{itemize}
\end{definition}

The next lemma demonstrates the relevance of these characteristics to our results 
by showing that, if $\mathrm{cof}^*(\mc{I}) = \mathrm{non}^*(\mc{I})$, then 
$\mc{I}$ carries an ascending $\subseteq^*$-cofinal sequence.

\begin{lemma}
  Suppose that $\mc{I}$ is an ideal on a set $X$ properly extending the ideal of 
  all finite subsets of $X$ with no $\sub^*$-maximal set,
  and suppose that $\mathrm{cof}^*(\mc{I}) = \mathrm{non}^*(\mc{I}) = \kappa$. Then there is an 
  ascending sequence $\vec{x} = \langle x_\alpha \mid \alpha < \kappa \rangle$ of elements of 
  $\mc{I}$ that is cofinal in $(\mc{I}, \subseteq^*)$.
\end{lemma}

\begin{proof}
  Let $\langle z_\alpha \mid \alpha < \kappa \rangle$ enumerate a $\subseteq^*$-cofinal 
  subset of $\mc{I}$.
  We will simultaneously construct $\vec{x} = \langle x_\alpha \mid \alpha < \kappa \rangle$ and 
  $\vec{e} = \langle e_w \mid w \in \Fn^+(\kappa) \rangle$ by recursion on $\alpha$ and 
  $\max(\dom(w))$ such that each $x_\alpha$ is in $\mc{I}$ and such that 
  $\langle e_w \mid w \in \Fn^+(\kappa) \rangle$ and $\kappa$ witness that 
  $\vec{x}$ is ascending. We will ensure that, for all $\alpha < \kappa$, we 
  have $z_\alpha \subseteq x_\alpha$, which will in turn ensure that $\vec{x}$ 
  is $\subseteq^*$-cofinal in $\mc{I}$. 
  To this end, suppose that $\beta < \kappa$ and we have constructed 
  $\vec{x} \restriction \beta$ and $\vec{e} \restriction \Fn^+(\beta)$ in such a way that 
  $\vec{e} \restriction \Fn^+(\beta)$ and $\beta$ satisfy clauses (2)(a--c) of Definition 
  \ref{ascending_def} for $\vec{x} \restriction \beta$. 
  We describe how to choose $x_\beta$ and $\langle e_w \mid 
  w \in \Fn^+(\beta+1) \wedge \beta \in \dom(w) \rangle$. 
  
  First, let $W_\beta = \{w \in \Fn^+(\beta) \mid \vec{x} \restriction \beta \text{ respects } 
  w\}$, and let 
  \[
    \mc{E}_\beta = \{e_w \cap \bigcap_{\alpha \in a} x_\alpha \mid 
    w \in W_\beta, \ a \in [\beta]^{<\omega}, \text{ and } \dom(w) < a\}.
  \]
  By the inductive hypothesis, we know that $\mc{E}_\beta \subseteq [X]^\omega$. 
  Moreover, since $\beta < \kappa = \mathrm{non}^*(\mc{I})$, we can find $x_{\beta,0} \in 
  \mc{I}$ such that $e \cap x_{\beta,0}$ is infinite for all $e \in \mc{E}_\beta$. 
  Since $\beta < \kappa = \mathrm{cof}^*(\mc{I})$, we can find $x_{\beta,1} \in \mc{I}$ 
  such that, for all $a \in [\beta]^{<\omega}$, we have $x_{\beta,1} \not\subseteq^* 
  \bigcup_{\alpha \in a} x_\alpha$. 
  
  Let $x_\beta = x_{\beta,0} \cup x_{\beta,1} \cup z_\beta$. Our choice of $x_{\beta,0}$ ensures that 
  this maintains clause (2)(c) of Definition \ref{ascending_def}, and our choice of 
  $x_{\beta,1}$ ensures that this maintains clause (2)(b) of the same definition. 
  Now, for every $w \in \Fn^+(\beta + 1)$ such that $\beta \in \dom(w)$ and 
  $\vec{x} \restriction (\beta + 1)$ respects $w$, let $e_w$ be an arbitrary element of 
  $[d^{\vec{x} \restriction (\beta + 1)}_w]^{\aleph_0}$. This completes the inductive step 
  of the construction and hence the proof of the lemma.
\end{proof}

An ideal of particular interest to us is the ideal $\emptyset \times \mathrm{Fin}$, which 
can be concretely defined as an ideal on $\omega \times \omega$ in the following way. 
First, for all functions $f\colon\omega \ra \omega$, recall that $I_f = \{(k,n) \in \omega \times 
\omega \mid n < f(k)\}$. Then $\emptyset \times \mathrm{Fin}$ is the set 
\[
  \{x \subseteq \omega \times \omega \mid (\exists f \in {^\omega}\omega)(x \subseteq I_f)\}.
\]
Recall that $\mathfrak{d}$ denotes the \emph{dominating number}, i.e., the least size of a cofinal 
subset of $({^\omega}\omega, \leq^*)$.

\begin{proposition} \label{prop: d_prop}
  $\mathfrak{d} = \mathrm{cof}^*(\emptyset \times \mathrm{Fin}) 
  = \mathrm{non}^*(\emptyset \times \mathrm{Fin})$.
\end{proposition}

\begin{proof}
  The fact that $\mathfrak{d} = \mathrm{cof}^*(\emptyset \times \mathrm{Fin})$ follows 
  immediately from the definitions. Let us show that $\mathfrak{d} = \mathrm{non}^*(\emptyset 
  \times \mathrm{Fin})$. First, suppose that $\mc{F} \subseteq {^\omega}\omega$ has size
  $\mathfrak{d}$ and is cofinal in 
  $({^\omega}\omega, \leq^*)$. For each $f \in {^\omega}\omega$, let $I^+_f$ denote the graph of 
  $f$, i.e., the set $\{(k,f(k)) \mid k < \omega\}$. Let $\mc{E} = \{I^+_f \mid f \in \mc{F}\}$. 
  We claim that $\mc{E}$ witnesses that $\mathrm{non}^*(\emptyset \times \mathrm{Fin}) \leq 
  \mathfrak{d})$. To see this, fix an arbitrary $x \in \emptyset \times \mathrm{Fin}$, and 
  find $g \in {^\omega}\omega$ such that $x \subseteq I_g$. 
  Then $I^+_g \in \mc{E}$ and $I_g \cap I^+_g = \emptyset$. Thus, 
  $\mathrm{non}^*(\emptyset \times \mathrm{Fin}) \leq \mathfrak{d}$.
  
  For the other inequality, suppose that $\mc{E}$ is a family of infinite elements of 
  $\emptyset \times \mathrm{Fin}$ with $|\mc{E}| < \mathfrak{d}$. We must find $x \in 
  \emptyset \times \mathrm{Fin}$ such that $|e \cap x| = \aleph_0$ for all $e \in \mc{E}$.
  For each $e \in \mc{E}$, fix $f_e \in {^\omega}\omega$ such that $e \subseteq I_{f_e}$. 
  Let $a_e = \{k < \omega \mid (\exists m < \omega)((k,m) \in e)\}$. Since $e$ is an infinite 
  element of $\emptyset \times \mathrm{Fin}$, $a_e$ must be infinite. Moreover, $a_e$ 
  partitions $\omega$ into pairwise disjoint intervals $\langle J^e_n \mid n < \omega \rangle$, 
  where $J^e_0 = [0, a_e(0)]$ and, for all $n < \omega$, $J^e_{n+1} = (a_e(n), a_e(n+1)]$. 
  Define a function $g_e \in {^\omega}\omega$ as follows. For all $k < \omega$, let $n < \omega$ 
  be such that $k \in J^e_n$, and then let $g_e(k) = f_e(a_e(n))$.
  
  Using the fact that $|\mc{E}| < \mathfrak{d}$, find $h \in {^\omega}\omega$ such that, 
  for all $e \in \mc{E}$, we have $h \not\leq^* g_e$. By increasing $h$ if necessary, we 
  may assume that it is weakly increasing, i.e., $h(k_0) \leq h(k_1)$ for all $k_0 < k_1 < \omega$.
  For each $e \in \mc{E}$, let $b_e := \{k < \omega \mid g_e(k) < h(k)\}$, and let 
  $c_e := \{n < \omega \mid b_e \cap J^e_n\} \neq \emptyset$. Since each $J^e_n$ is finite, we 
  have $c_e \in [\omega]^\omega$. Moreover, for all $n \in c_e$, we can fix $k \in b_e 
  \cap J^e_n$ and conclude that
  \[
    f_e(a_e(n)) = g_e(k) < h(k) \leq h(a_e(n)),
  \]
  and hence there is $m \leq h(a_e(n))$ such that $(a_e(n),m) \in e$. In particular, 
  letting $x = I_h$, it follows that $x \in \emptyset \times \mathrm{Fin}$ and 
  $|e \cap x| = \aleph_0$ for all $e \in \mc{E}$.
\end{proof}

Putting the results of this section together with Theorems \ref{thm: ascending} and 
\ref{thm: ascending_wide} immediately yields the following corollary.

\begin{corollary}
  Suppose that $1 \leq n < \omega$, $\mc{I}$ is an ideal on a set $X$ properly extending 
  the ideal of all finite subsets of $X$ with no $\sub^*$-maximal set. Suppose also 
  that $\mathrm{cof}^*(\mc{I}) = \mathrm{non}^*(\mc{I}) = \omega_n$. Then
  \begin{enumerate}
    \item $\lim^n \mathbf{A}_{\mc{I}}[\bb{Z}^{(\omega_n)}] \neq 0$;
    \item if $\wkd(S^{k+1}_k)$ holds for all $k < n$, then $\lim^n \mb{A}_{\mc{I}} \neq 0$.
  \end{enumerate}
\end{corollary}

Setting $\mc{I} = \emptyset \times \mathrm{Fin}$ and invoking Proposition \ref{prop: d_prop}
and Fact \ref{fact: cofinal_lim} then yields the following important special case, which is 
Theorem A from the introduction.

\begin{corollary}
  Suppose that $1 \leq n < \omega$ and $\mathfrak{d} = \omega_n$. Then
  \begin{enumerate}
    \item $\lim^n \mathbf{A}[\bb{Z}^{(\omega_n)}] \neq 0$;
    \item if $\wkd(S^{k+1}_k)$ holds for all $k < n$, then $\lim^n \mathbf{A} \neq 0$.
  \end{enumerate}
  In particular, if $\lim^n \mb{A}[H] = 0$ for all $1 \leq n < \omega$ and all abelian 
  groups $H$, then $2^{\aleph_0} > \aleph_\omega$.
\end{corollary}

We end this section by noting one other important special case. Given a regular uncountable 
cardinal $\kappa$, let $\mc{I}_\kappa$ denote the ideal $[\kappa]^{<\kappa}$. It is easily 
verified that $\mathrm{cof}^*(\mc{I}_\kappa) = \mathrm{non}^*(\mc{I}_\kappa) = \kappa$. 
Our framework therefore immediately yields the following corollary.

\begin{corollary} \label{cor: omega_n}
  Suppose that $1 \leq n < \omega$. Then
  \begin{enumerate}
    \item $\lim^n \mb{A}_{\mc{I}_{\omega_n}}[\bb{Z}^{(\omega_n)}] \neq 0$;
    \item if $\wkd(S^{k+1}_k)$ holds for all $k < n$, then $\lim^n \mb{A}_{\mc{I}_{\omega_n}} \neq 0$.
  \end{enumerate}
\end{corollary}

This corollary is not new, as the inverse systems $\mb{A}_{\mc{I}_\kappa}[H]$ have been somewhat 
extensively studied, although under different names (for instance, in \cite{svhdlwlc} they 
are denoted $\mb{C}(\kappa, H)$). The first derived limits of these systems are also 
familiar to set theorists, although in a different guise: the assertion ``$\lim^1 
\mb{A}_{\mc{I}_\kappa} \neq 0$" is equivalent to the existence of a coherent 
$\kappa$-Aronszajn subtree of ${^{<\kappa}}\omega$.

Clause (1) of Corollary \ref{cor: omega_n} is implicit in 
\cite{mitchell}; for a more set-theoretic proof, see \cite{bergfalk_alephs}. Clause (2) has, 
to the best of our knowledge, not explicitly appeared anywhere, but it follows from the 
techniques in \cite{velickovic2021non}. We include the corollary here to indicate that these 
systems can be incorporated into the general framework we develop in this paper.

\section{Simultaneous nonvanishing from squares and weak diamonds} \label{sec: squares}

In this final section, we prove Theorem B from the introduction. We first recall the following 
definition.

\begin{definition}

Suppose that $\lambda$ is a regular uncountable cardinal and $S \sub \lambda$ is stationary. The principle $\square(\lambda, S)$ asserts the existence of a sequence $\vec{C} = \langle C_\alpha \mid \alpha < \lambda \rangle$ such that: 

\begin{enumerate}

\item for every limit ordinal $\alpha < \lambda$, $C_\alpha$ is a club in $\alpha$;

\item for all limit ordinals $\alpha < \beta < \lambda$, if $\alpha \in \Lim(C_\beta)$, then $C_\beta \cap \alpha = C_\alpha$;

\item for every limit ordinal $\alpha < \lambda$, $C_\alpha \cap S = \emptyset$.

\end{enumerate}
Such a sequence is called a $\square(\lambda, S)$-sequence.
\end{definition}

\begin{remark}

Condition $(3)$ implies that there exists no club $D \sub \lambda$ such that for all  $\alpha \in \Lim(D)$, $D \cap \alpha = C_\alpha$, i.e., there is no \emph{thread} 
through $\vec{C}$. For a proof, if it were not so, let $\gamma \in D \cap S$ and $\alpha \in \Lim(D) \setminus (\gamma +1)$. Then $\gamma \in D \cap \alpha = C_\alpha$, a contradiction. In particular, if $S \subseteq \lambda$ is stationary, then a 
$\square(\lambda, S)$-sequence is \emph{a fortiori} a $\square(\lambda)$-sequence.

\end{remark}

If $\delta$ is an ordinal, then a \emph{$\delta$-chain} in $({^{\omega}}\omega, \leq^*)$ 
is a sequence $\vec{f} = \langle f_\alpha \mid \alpha < \delta \rangle$ such that, 
for all $\alpha < \beta < \delta$, we have $f_\alpha <^* f_\beta$. If $H$ is an abelian 
group and $\vec{f}$ is a $\delta$-chain in $({^{\omega}}\omega, \leq^*)$, then 
when we speak of an \emph{$H$-valued $n$-family indexed over $\vec{f}$}, we mean an 
$n$-family of the form 
\[
  \left\langle \varphi_a \colon \bigcap_{\alpha \in b} I_{f_\alpha} \ra H 
  ~ \middle\vert ~ a \in [\delta]^n \right\rangle.
\]

\begin{lemma}
\label{lemma_stepup}

Suppose that $\kappa < \lambda  $ are uncountable regular cardinals, $n$ is a positive integer, and $H$ is an abelian group with $\vert H \vert \leq 2^\kappa$. Suppose moreover that 

\begin{enumerate}
\item for every $\kappa$-chain in $({}^\omega \omega, \leq^\ast)$ there exists a nontrivial coherent $H$-valued $n$-family indexed over that chain;

\item there exists a stationary $S \sub S^\lambda_\kappa$ such that both $\wl(S)$ and $\square(\lambda, S)$ hold.
\end{enumerate}

Then for every $\lambda$-chain in $({}^\omega \omega, \leq^\ast)$ there exists a nontrivial coherent $H$-valued $(n+1)$-family indexed over that chain.

\end{lemma}

\begin{proof}

Fix a $\lambda$-chain $\vec{f} = \langle f_\alpha \mid \alpha < \lambda \rangle$ in $({^{\omega}}\omega, \leq^*)$ and a $\square(\lambda, S)$-sequence $\vec{C} = \langle C_\alpha \mid \alpha < \lambda \rangle$. Let $\mc{S}$ be the set of $\sigma \in {^{<\lambda}}2$ such 
that $\supp(\sigma) \sub S$, and let $\mc{S}^+$ be the set of 
$g \in {^{\lambda}}2$ such that $\supp(g) \sub S$. For all $\sigma \in \mc{S}$, we will construct a coherent (and trivial) $(n+1)$-family $\Phi^\sigma = \langle \varphi^\sigma_b \colon \bigcap_{\alpha \in b} I_{f_{\alpha}} \to H \mid b \in [\lh(\sigma)]^{n+1} \rangle$. We will arrange so that, if $\sigma, \tau \in \mc{S}$ and $\sigma 
\sqsubseteq \tau$, then $\Phi^\sigma \sqsubseteq \Phi^\tau$. At the end of the proof, 
we will find $g \in \mc{S}$ such that $\Phi^g = \bigcup_{\alpha < \lambda} \Phi^{g 
\restriction \alpha}$ is nontrivial.
The construction is by recursion on $\lh(\sigma)$, and we will use $\vec{C}$ to ensure coherence by maintaining the following recursion hypothesis for each $\sigma \in \mc{S}$:

\[ (\dagger) \,\,\,\,\,\,\,\,\, \forall\gamma \in 
\mathrm{Lim}(\lh(\sigma)) \setminus S ~ \forall\beta \in \mathrm{Lim}(C_\gamma) ~ 
\forall a \in [\beta]^n \left[\varphi^\sigma_{a \cup \{\beta\}} 
=^* \varphi^\sigma_{a \cup \{\gamma\}} \right]  \]

We now turn to the construction. First, if $\sigma \in \mc{S}$,
$\lh(s) = \delta \in \Lim(\lambda)$, and $\Phi^{\sigma \restriction \alpha}$ 
has been constructed for all $\alpha < \delta$, then let $\Phi^\sigma = \bigcup_{\alpha < \delta} \Phi^{\sigma \restriction \alpha}$.

It remains to describe the successor step of the construction. Thus, suppose 
that $\sigma \in \mc{S}$, $\lh(\sigma) = \delta$, and we have constructed $\Phi^\sigma$. 
We describe how to construct $\Phi^{\sigma, \ell}$ for $\ell < 2$ 
if $\delta \in S$ and $\ell = 0$ otherwise.
We distinguish between a number of cases.

\textbf{Case 1: $\delta$ is a successor ordinal}. In this case there is nothing new to check with respect to $(\dagger)$, so any coherent extension $\Phi^{\sigma,0}$ of $\Phi^\sigma$ will do. Such an extension exists by Theorem \ref{prop: goblot_var} and Remark 
\ref{remark: coherence}.

\textbf{Case 2: $\delta \in \Lim(\lambda) \setminus S$ and $\delta > \sup(\Lim(C_\delta))$}. Let $\gamma = \sup(\Lim(C_\delta))$. In the present case, $\gamma = \max(\Lim(C_\delta))$ and $\otp(C_\delta \setminus \gamma) = \omega$, thus witnessing that $\cf(\delta) = \omega$. Hence again by Theorem \ref{prop: goblot_var} there exists $\Psi^\sigma = \langle \psi^\sigma_a \mid a \in [\delta]^n \rangle$ that trivializes $\Phi^\sigma$. 

In order to construct $\Phi^{\sigma, 0} $, we need to define 
$\varphi^{\sigma,0}_{a \cup \{\delta\}}$ for all $a \in [\delta]^n$. 
Before giving the details, let us explain the basic idea. In order to maintain 
$(\dagger)$, we attempt to make $\varphi^{\sigma,0}_{a \cup \{\delta\}}$ resemble 
$\varphi^\sigma_{a \cup \{\gamma\}}$ as much as possible while maintaining coherence.
This will entail setting $\varphi^{\sigma,0}_{a \cup \{\delta\}} \restriction I_{f_\gamma}$ 
equal to $\varphi^\sigma_{a \cup \{\gamma\}} \restriction I_{f_\delta}$ multiplied by an 
appropriate power of $-1$ (to maintain coherence). We are then left with defining 
$\varphi^{\sigma,0}_{a \cup \{\delta\}}$ outside of $I_{f_\gamma}$. Here we can extend in an 
arbitrary way (while maintaining coherence); we will use the trivialization $\Psi^\sigma$ 
to do this.

Now fix $a \in [\delta]^n$. If $\gamma \in a$, then set 
$\varphi^{\sigma,0}_{a \cup \{\delta\}} = 0$.
If $\gamma \notin a$, then let $h = |a \setminus \gamma|$, and define 
$\varphi^{s,0}_{a \cup \{\delta\}} \colon \bigcap_{\alpha \in a} I_{f_{\alpha}} \cap I_{f_\delta} 
\to H$ by setting

\[
	\varphi^{s,0}_{a \cup \{\delta\}}(x) = \begin{cases}
	  (-1)^h \varphi^\sigma_{a \cup \{\gamma\}}(x) & \text{if } x \in I_{f_\gamma}; \\
	  (-1)^n \psi^\sigma_a(x) & \text{if } x \notin I_{f_\gamma}.
	\end{cases}
\]

Let us verify that this choice of $\Phi^{\sigma,0}$ satisfies the requirements of the 
construction. For readability, we will omit the superscripts $\sigma$ and $0$ during 
this verification.
To verify condition $(\dagger)$, we must show that, for all $\beta \in \Lim(C_\delta)$ 
and all $a \in [\beta]^n$, we have $\varphi_{a \cup \{\beta\}} =^* 
\varphi_{a \cup \{\delta\}}$. Fix such $\beta$ and $a$. By the 
construction of $\Phi^{s,0}$, we know that $\varphi_{a \cup \{\gamma\}} 
=^* \varphi_{a \cup \{\delta\}}$. By the properties of $\vec{C}$, we know 
that either $\beta = \gamma$ or $\beta \in \Lim(C_\gamma)$. In the latter case, the 
fact that $\Phi^\sigma$ satisfies $(\dagger)$ implies that 
$\varphi_{a \cup \{\beta\}} =^* \varphi_{a \cup \{\gamma\}}$. Since 
$I_{f_\beta} \sub^* I_{f_\gamma} \sub^* I_{f_\delta}$, this in turn implies 
that $\varphi_{a \cup \{\beta\}} =^* 
\varphi_{a \cup \{\delta\}}$, as desired.

We next verify coherence. It suffices to show that, for all $b \in [\delta]^{n+1}$, 
we have $(-1)^{n+1} \varphi_b + \sum_{i < n+1} (-1)^i \varphi_{b^i \cup \{\delta\}} 
=^* 0$. Fix such a $b$. We distinguish between two cases. First suppose that 
$\gamma \in b$ (say $\gamma = b(k)$, so $|b \setminus \gamma| = n - k$). Then
\begin{align*}
  (-1)^{n+1} \varphi_b + \sum_{i < n+1} (-1)^i \varphi_{b^i \cup \{\delta\}} & 
  = (-1)^{n+1} \varphi_b + (-1)^k\varphi_{b^k \cup \{\delta\}} \\ 
  &= (-1)^{n+1} \varphi_b + (-1)^k (-1)^{n-k} \varphi_{b^k \cup \{\gamma\}} \\ 
  &= (-1)^{n+1} \varphi_b + (-1)^n \varphi_b = 0.
\end{align*}

Suppose on the other hand that $\gamma \notin b$, and let $k = |b \cap \gamma|$ 
(so $|b \setminus \gamma| = n+1-k$). Then, when restricting to $I_{f_\gamma}$, 
we have
\begin{align*}
  & (-1)^{n+1} \varphi_b + \sum_{i < n+1} (-1)^i \varphi_{b^i \cup \{\delta\}} \\
  & = (-1)^{n+1} \varphi_b + \sum_{i < k} (-1)^i(-1)^{n+1-k} \varphi_{b^i \cup \{\gamma\}} 
   + \sum_{i = k}^n (-1)^i(-1)^{n-k} \varphi_{b^i \cup \{\gamma\}} \\ 
  & = (-1)^{n+1-k} \Big( (-1)^k \varphi_b + \sum_{i<k}(-1)^i \varphi_{b^i \cup \{\gamma\}} 
   + \sum_{i=k}^n (-1)^{i+1} \varphi_{b^i \cup \{\gamma\}}\Big) \\ 
  & = (-1)^{n+1-k} \sum_{i < n+2} (-1)^i \varphi_{(b \cup \{\gamma\})^i} =^* 0.
\end{align*}
Outside of $I_{f_\gamma}$, we have 
\begin{align*}
  (-1)^{n+1} \varphi_b + \sum_{i < n+1} (-1)^i \varphi_{b^i \cup \{\delta\}} & = 
  (-1)^{n+1} \varphi_b + \sum_{i < n+1} (-1)^i(-1)^n \psi_{b^i} \\ 
  & =^* (-1)^{n+1} \varphi_b + (-1)^n \varphi_b = 0.
\end{align*}
Thus, we have verified coherence.

\textbf{Case 3:
$ \delta \in \Lim(\lambda) \setminus S $ and $\sup(\Lim(C_\delta)) = \delta$.}  For $a \in [\delta]^n $, let $\gamma(a) = \min(\Lim(C_\delta) \setminus (\max(a) + 1))$. Then, 
for all $x \in I_{f_\delta} \cap \bigcap_{\alpha \in a} I_{f_\alpha}$, set
\[
  \varphi^{\sigma,0}_{a \cup \{\delta\}} = \begin{cases}
    \varphi^\sigma_{a \cup \{\gamma(a)\}}(x) & \text{if } x \in I_{f_{\gamma(a)}} \\ 
    0 & \text{otherwise}.
  \end{cases}
\]
Note that the ``otherwise" case of the above definition only occurs for finitely many 
$x \in \dom(\varphi^{\sigma,0}_{a \cup \{\delta\}})$.

We verify condition $(\dagger)$, again omitting superscripts for readability. For all  $\beta \in \Lim(C_\delta)$ and $a \in [\beta]^n  $ we have $\varphi_{a \cup \{\delta\}} =^\ast \varphi_{a \cup \{\gamma(a)\}} =^\ast \varphi_{a \cup \{\beta\}}$, where the second equality holds by the induction assumption since either $\gamma(a) = \beta$ or $\gamma(a) \in \Lim(C_\beta) = \Lim(C_\delta) \cap \beta $. Since moreover $I_{f_{\gamma(a)}} 
\sub^* I_{f_\beta} \sub^* I_{f_\delta}$, we have $\varphi_{a \cup \{\delta\}} =^* 
\varphi_{a \cup \{\beta\}}$, as desired.

We now verify coherence. Let $b \in [\delta]^{n+1}$, and let $\gamma = 
\gamma(b^0)$. Then
\begin{align*}
  (-1)^{n+1} \varphi_b + \sum_{i < n+1} (-1)^i \varphi_{b^i \cup \{\delta\}} & =^* 
  (-1)^{n+1} \varphi_b + \sum_{i < n+1} (-1)^i \varphi_{b^i \cup \{\gamma\}} \\ 
  & =^* \sum_{i < n+2} (-1)^i \varphi_{(b \cup \{\gamma\})^i} =^* 0,
\end{align*}
where the first equality (mod finite) uses the fact that $\gamma(b^i) = \gamma$ for 
all $i < n$ and $\varphi_{b^n \cup \{\gamma\}} =^* \varphi_{b^n \cup \{\gamma(b^n)\}}$ 
by $(\dagger)$, and the last equality holds because $\Phi^\sigma$ satisfies coherence.

\textbf{Case 4:} $\delta \in S$. In this case there is nothing to check with respect to $(\dagger)$, so we only need to produce two coherent families $\Phi^{\sigma, 0}$ and $\Phi^{\sigma,1}$. First observe that $\sup(\Lim(C_\delta)) = \delta$, since 
$\delta \in S \sub S^\lambda_\kappa$. We can thus define $\Phi^{\sigma,0}$ exactly 
as in Case 3. 

By the hypothesis of the theorem, every $\kappa$-chain in $({^\omega}\omega, \leq^*)$ 
carries a nontrivial coherent $H$-valued $n$-family. By Proposition \ref{prop: extension} 
and Remark \ref{remark: extension}, the same is true for every chain whose length 
has cofinality $\kappa$. In particular, we can fix a nontrivial coherent $H$-valued 
$n$-family 
\[
  T^\delta = \langle \tau^\delta_a \colon \bigcap_{\alpha \in a} I_{f_\alpha} \ra H 
  \mid a \in [\delta]^n \rangle.
\]
Note that, since $I_{f_\alpha} \subseteq^* I_{f_\delta}$ for all $\alpha < \delta$, 
the restriction $T^\delta \restriction \restriction I_{f_\delta}$ remains nontrivial.
Then define $\Phi^{\sigma,1}$ by setting $\varphi^{\sigma,1}_{a \cup \{\delta\}} 
= \varphi^{\sigma,0}_{a \cup \{\delta\}} + \tau^\delta_a$ for all $a \in [\delta]^n$.
The coherence of $\Phi^{\sigma,1}$ then follows from the coherence of 
$\Phi^{\sigma,0}$ and $T^\delta$.

This completes the construction. We now prove that there is $g \in \mc{S}^+$ 
such that $\Phi^g = \bigcup_{\beta < \lambda} \Phi^{g \restriction \beta}$ is 
nontrivial. We first record a couple of simple claims.

\begin{claim}
\label{claim_twisting}
Suppose that $\delta \in S$ and $\sigma \in {^\delta}2$. Then no trivialization of 
$\Phi^\sigma$ extends both to a trivialization of $\Phi^{\sigma,0}$ and to
a trivialization of $\Phi^{\sigma,1}$.
\end{claim}

\begin{proof}
This follows from the construction of $\Phi^{\sigma,0}$ and $\Phi^{\sigma,1}$ and 
from Propositions \ref{nonequiv_prop} and \ref{nonextension_prop}.
\end{proof}

\begin{claim}
\label{claim_atmostone}

Let $\Upsilon$ be a $n$-family over $\vec{f} \restriction \beta$ for some $\beta < \lambda$. Then there exists at most one $\sigma \in {}^\beta 2$ such that $\Upsilon$ trivializes $\Phi^\sigma$. 
\end{claim}

\begin{proof}
This is proven in exactly the same way as Claim \ref{claim: unique_triv}.
\end{proof}

Now we use a coding similar to that used in the proof of Theorem \ref{thm: ascending}. More precisely, let $\mathcal{P}^\gamma$ be the set of alternating $H$-valued $n$-families indexed over $\vec{f} \restriction \gamma$. Then let $G$ be a function with domain ${}^{\leq \lambda} 2$ such that for every $\gamma \in S^\lambda_\kappa \cup \lbrace \lambda \rbrace$ the map $G \restriction {}^\gamma 2$ is a surjection onto $\mathcal{P}^\gamma$, and such 
that if $\rho \sqsubseteq \rho' \in \dom(G)$, then $G(\rho) \sqsubseteq 
G(\rho')$. Such a $G$ exists because the intervals between successive elements of $S^\lambda_\kappa$ have size $\kappa$, and $\vert H \vert \leq 2^\kappa$. We also write $\Upsilon^\rho$ for $G(\rho)$. 

Finally, we define $F\colon {}^{< \lambda} 2 \to 2$ as follows. If $\lh(\rho) \in S$ and $\Upsilon^\rho$ trivializes $\Phi^\sigma$ for some $\sigma$ with $\lh(\sigma) = \lh(\rho)$ (such a $\sigma$ is unique by Claim \ref{claim_atmostone}) then we let $F(\rho) = \ell \in \lbrace 0 , 1 \rbrace$ be such that $\Upsilon^\rho$ does not extend to a trivialization of $\Phi^{\sigma, \ell}$ (such an $\ell$ exists by Claim \ref{claim_twisting}). Otherwise, let $F(\rho) \in \lbrace 0, 1 \rbrace$ be arbitrary.

By $\wl(S)$ there exists $g\colon \lambda \to 2$ such that for all $b\colon \lambda \to 2$ there exists $\alpha \in S$ such that $g(\alpha) = F(b \restriction \alpha)$. Since this only 
depends on the values that $g$ takes on $S$, we can assume that $g \in \mc{S}^+$.
Now suppose for the sake of contradiction that some $\Upsilon$ trivializes $\Phi^g = \bigcup_{\alpha < \lambda} \Phi^{g \restriction \alpha}$. Find $b$ such that 
$\Upsilon = \Upsilon^b$, and find $\alpha \in S$ such that $g(\alpha) = F(b \restriction 
\alpha)$. Then $\Upsilon^b \restriction (\vec{f} \restriction \alpha) = 
\Upsilon^{b \restriction \alpha}$ trivializes $\Phi^g \restriction (\vec{f} \restriction 
\alpha) = \Phi^{g \restriction \alpha}$. By our choice of $F$, $g$, and $\alpha$, 
then, no extension of $\Upsilon^{b \restriction \alpha}$ can trivialize 
$\Phi^{g \restriction (\alpha + 1)}$ and hence, \emph{a fortiori}, no extension of 
$\Upsilon^{b \restriction \alpha}$ can trivialize $\Phi^g$, contradicting the 
fact that $\Upsilon$ extends $\Upsilon^{b \restriction \alpha}$ and trivializes 
$\Phi^g$ and completing the proof of the theorem.
\end{proof}

We can now show the consistency of the simultaneous nonvanishing of $\lim^k \mb{A}$ 
for many values of $k$. 
We first show that for all $n$ it is consistent that $	 \Et_{2 \leq k \leq n} \lim^k \textbf{A} \neq 0$. This can be realized in the kind of model considered in \cite{velickovic2021non}. Recall that, for a regular cardinal $\kappa$ and a set 
$I$, the poset $\mathrm{Fn}(I,2,\kappa)$ consists of all partial functions 
$p \colon I \ra 2$ such that $|p| < \kappa$, ordered by reverse inclusion. Recall also that 
Hechler forcing $\bb{H}$ consists of conditions of the form $p = (s_p, f_p)$ such 
that $s_p \in {^{<\omega}}\omega$ and $f_p \in 
{^\omega}\omega$. If $p,q \in \bb{H}$, then $q \leq p$ if and only if 
$s_q \supseteq s_p$, $f_q \geq f_p$ and, for all $n \in \dom(s_q) \setminus \dom(s_p)$, we 
have $s_q(n) \geq f_p(n)$.

\begin{lemma}
\label{lemma_squaresomegan}
Suppose that $1 \leq n < \omega$. Then there exists a model of $\ZFC$ in which 
\begin{enumerate}
  \item $\mathfrak{b} = \mathfrak{d} = \omega_n$;
  \item for all $k < n$, $\wl(S^{k+1}_k)$ holds;
  \item for all $k < n$, there exists a stationary set $S \subseteq S^n_k$ 
  such that $\square(\omega_n, S) + \wl(S)$ holds.
\end{enumerate}
\end{lemma}

\begin{proof}
Our model will be the following forcing extension of $L$:
\[
  V' = L^{\bb{H}_{\omega_n} \times \bb{C}_1 \times \cdots \times \bb{C}_n},
\]
where $\bb{H}_{\omega_n}$ is the finite support iteration of
Hechler forcings of length $\omega_n$ and $\bb{C}_k = \textrm{Fn}(\omega_{n+k}, 2, \omega_k)$ for all $1 \leq k \leq n$, i.e., $\bb{C}_k$ is the forcing to add $\omega_{n+k}$-many Cohen subsets to 
$\omega_k$. 

It follows from  \cite[Theorem 4.5]{velickovic2021non} that $V'$ has all of the same 
cardinals and cofinalities as $L$ and 
satisfies $\mathfrak{b}= \mathfrak{d}= \omega_n$ as well as $\wl(S^{k+1}_k)$ for all
$k < n$. To verify item (3) in the statement of the lemma, fix $k < n$, 
and temporarily work in $L$. By \cite{jensen}, we can find a stationary subset 
$S \subseteq S^n_k$ and a $\square(\omega_n, S)$-sequence $\vec{C}$. 
Let $\bb{P} = \bb{H}_{\omega_n} \times \bb{C}_1 \times \cdots \times \bb{C}_{n-1}$. 
Then $\bb{P}$ is an $\omega_n$-cc poset of size $\omega_{2n-1}$, $\bb{C}_n$ 
is $\omega_n$-closed, and $V' = L^{\bb{P} \times \bb{C}_n}$. In particular, 
$S$ remains stationary in $V'$. Since all of the defining properties of 
$\square(\omega_n, S)$ are upwards absolute from $L$ to $V'$, $\vec{C}$ remains a 
$\square(\omega_n, S)$-sequence in $V'$. Moreover, \cite[Lemma 4.3]{velickovic2021non}
implies that $\wl(S)$ holds in $V'$, so $S \sub S^n_k$ is as desired.
\end{proof}

\begin{lemma}
\label{lemma_nontrivialchain}
Fix $1 \leq n < \omega$. Assume that 
\begin{enumerate}
  \item for all $k < n$, $\wl(S^{k+1}_k)$ holds;
  \item for all $1 \leq k < n$, there exists a stationary set $S \subseteq S^n_k$ 
  such that $\square(\omega_n, S) + \wl(S)$ holds.
\end{enumerate}
Then every $\omega_n$-chain in $({}^\omega \omega, \leq^\ast)$ carries a nontrivial coherent $H$-valued $k$-family, for all $2 \leq k \leq n$ and all countable abelian groups $H$.    
\end{lemma}

\begin{proof}
Fix a countable abelian group $H$.
In $\ZFC$, every $\omega_1$-chain in $({}^\omega \omega, \leq^\ast)$ carries a nontrivial coherent $H$-valued $1$-family (cf.\ \cite[Proposition 3.4]{velickovic2021non}). Using \cite[Lemma 3.6]{velickovic2021non}, one can prove by induction that, for 
every $1 \leq k \leq n$, every $\omega_k$-chain carries a nontrivial coherent 
$H$-valued $k$-family. Finally, suppose that $2 \leq k \leq n$, and pick
$S \subseteq S^{\omega_n}_{\omega_{k-1}}$ such that $\square(\omega_n, S) + 
\wl(S)$ holds. Then Lemma \ref{lemma_stepup} implies that every $\omega_n$-chain 
carries a nontrivial coherent $H$-valued $k$-family, as desired.
\end{proof}

The following theorem is now immediate, yielding clause (1) of Theorem B from the introduction.

\begin{theorem} \label{thm: n_simul}
Fix $2 \leq n < \omega$.
Relative to the consistency of $\ZFC$, it is consistent that $\mathfrak{b} = \mathfrak{d} 
= \omega_n$ and
$\Et_{2 \leq k \leq n} \lim^k \textbf{A} \neq 0$ holds. 
\end{theorem}

\begin{proof}
  In any model witnessing the conclusion of Lemma \ref{lemma_squaresomegan}, 
  there exists an $\omega_n$-chain that is cofinal in $({^{\omega}}\omega, \leq^*)$. 
  Lemma \ref{lemma_nontrivialchain} then implies that, for all $2 \leq k \leq n$, 
  this $\omega_n$-chain carries a nontrivial coherent $\bb{Z}$-valued $k$-family.
  Then Proposition \ref{prop: cofinal_subset}, Remark \ref{remark: extension}, 
  and Fact \ref{fact: lim_coh} imply that $\lim^k \mathbf{A} \neq 0$ for all $2 \leq k \leq n$.
\end{proof}

We now show that it is consistent that $\lim^n \mathbf{A}$ does not vanish for any $n \geq 2$.
By Goblot's Theorem (cf.\ Proposition \ref{prop: goblot_var}), this requires the dominating number 
to be at least $\aleph_{\omega+1}$. For technical reasons, we were unable to obtain this simultaneous 
nonvanishing with $\mathfrak{d} = \aleph_{\omega+1}$, but we were able to achieve it with
$\mathfrak{d} = \aleph_{\omega + 2}$.

\begin{lemma}
\label{lemma:long-chain-consistency}
There exists a model of $\ZFC$ in which 
\begin{enumerate}
  \item $\mathfrak{b} = \mathfrak{d} = \omega_{\omega+2}$;
  \item for all $k < \omega$, $\wl(S^{k+1}_k)$ holds;
  \item for all $k < \omega$, there exists a stationary set $S \subseteq 
  S^{\omega_{\omega+2}}_{\omega_k}$ such that $\square(\omega_{\omega+2}, S) + \wl(S)$ holds.
\end{enumerate}
\end{lemma}

\begin{proof}
Our model will be the following forcing extension of $L$:
\[  
	V' = L^{\mathbb{H}_{\omega_{\omega+2}} \times \prod_{1 \leq n < \omega} \mathbb{C}_n \times \bb{C}_\omega }         
\]
satisfies the properties in question,  where $\mathbb{H}_{\omega_{\omega+2}}$ is the finite-support iteration of Hechler forcings of length $\omega_{\omega+2}$, $\mathbb{C}_n= \mathrm{Fn}(\omega_{\omega + n + 2}, 2, \omega_n)$ for all $1 \leq n < \omega$, and $\bb{C}_\omega = \mathrm{Fn}(\omega_{\omega \cdot 2 + 2} , 2, \omega_{\omega+2})$.

We first verify clause (1) in the statement of the theorem. Let $G$ be a generic filter for 
$\mathbb{H}_{\omega_{\omega+2}} \times \prod_{1 \leq n < \omega} \mathbb{C}_n \times \bb{C}_\omega$. By standard 
arguments, forcing over $L$ with this product preserves all cardinals and cofinalities. 
We can write $G$ as the product $G= G_0 \times G'$ of generic filters on $\bb{H}_{\omega_{\omega+2}}$ and $\prod_{1 \leq n < \omega} \mathbb{C}_n \times \bb{C}_\omega $, respectively. Note that $\prod_{1 \leq n < \omega} \mathbb{C}_n \times \bb{C}_\omega$ is $\sigma$-closed, so that in particular it does not add any reals, preserves $\CH$ and $\bb{H}_{\omega_{\omega+2}}$ is absolute with respect to it. Therefore, forcing with 
$\bb{H}_{\omega_{\omega+2}}$ over $L[G']$ forces that $\mathfrak{b} = \mathfrak{d} = \omega_{\omega+2}$, as desired.

We now verify clause (2). Let us further decompose $G'$ as a product $\prod_{1 	\leq n < \omega} G_n 
\times G_\omega$, where $G_n$ is generic for $\bb{C}_n$ for all $1 \leq n \leq \omega$. Fix 
some $k < \omega$. Then $\prod_{k+2 \leq n < \omega} \bb{C}_n \times \bb{C}_\omega$ is $\omega_{k+2}$-closed, so that $W = L[\prod_{k+2 \leq n < \omega} G_n \times G_\omega]$ satisfies $2^{\omega_k} = \omega_{k+1}$. Now apply \cite[Lemma 4.3]{velickovic2021non} in $W$ with $\bb{P} = \bb{H}_{\omega_{\omega+2}} \times \prod_{j \leq k} \bb{C}_k$ and $\bb{Q} = \bb{C}_{k+1}$ to conclude that $\wl(S)$ holds for every stationary $S \subseteq \omega_{k+1}$ with $S \in W$; in particular, for $S = S^{k+1}_{k}$. More 
precisely, in $W$, $\bb{P}$ is $\omega_{k+1}$-cc and has size $\omega_{\omega + k + 2}$, 
$\bb{Q} = \mathrm{Fn}(\omega_{\omega + k + 3}, 2, \omega_{k + 1})$ and, in $W$, we have 
$(\omega_{\omega + k + 2})^{\omega_k} = \omega_{\omega + k + 2}$, so 
\cite[Lemma 4.3]{velickovic2021non} implies that $\wl(S^{k+1}_k)$ holds in the extension of $W$ 
by $\bb{P} \times \bb{Q}$, i.e., in $V'$.

We finally verify clause (3). Fix $k < \omega$, and temporarily work in $L$. By \cite{jensen}, we can 
find a stationary subset $S \subseteq S^{\omega_{\omega+2}}_{\omega_k}$ and a 
$\square(\omega_{\omega+2}, S)$-sequence $\vec{C}$. Since $\mathbb{H}_{\omega_{\omega + 2}} \times \prod_{1 \leq n < \omega} 
\mathbb{C}_n$ is $\omega_{\omega+2}$-cc and $\bb{C}_\omega$ is $\omega_{\omega+2}$-closed, $S$ remains 
stationary in $V'$, so $\vec{C}$ remains a $\square(\omega_{\omega+2}, S)$-sequence in $V'$.
Another application of \cite[Lemma 4.3]{velickovic2021non}, this time with 
$\bb{P} = \mathbb{H}_{\omega_{\omega+2}} \times \prod_{1 \leq n < \omega} \mathbb{C}_n$ and $\bb{Q} = 
\mathbb{C}_\omega$, implies that $\wl(S)$ holds in $V'$.
\end{proof}

\begin{lemma}
\label{lemma:long-chain-implication}
Assume that $\lambda \geq \omega_{\omega+1}$ is a regular cardinal and
\begin{enumerate}
  \item for all $1 \leq k < \omega$, $\wl(S^{k+1}_k)$ holds;
  \item for all $1 \leq k < \omega$, there is a stationary set $S \subseteq S^\lambda_{\omega_k}$ 
  such that $\square(\lambda, S) + \wl(S)$ holds.
\end{enumerate}
Then every $\lambda$-chain in $({^{\omega}}\omega, \leq^*)$ carries a nontrivial coherent 
$H$-valued $k$-family for all $2 \leq k < \omega$ and all countable abelian groups $H$.
\end{lemma}

\begin{proof}
Fix a countable abelian group $H$. As above, we know that every $\omega_1$-chain in $({^{\omega}}\omega, 
\leq^*)$ carries a nontrivial coherent $H$-valued $1$-family. Using \cite[Lemma 3.6]{velickovic2021non}, one can prove by induction that, for all $1 \leq k < \omega$, every $\omega_k$-chain carries a nontrivial coherent $H$-valued $k$-family. Finally, fix $2 \leq k < \omega$ and fix a stationary 
$S \subseteq S^\lambda_{\omega_{k-1}}$ such that $\square(\lambda, S) + \wl(S)$ holds. Then 
Lemma \ref{lemma_stepup} implies that every $\lambda$-chain carries a nontrivial coherent
$H$-valued $k$-family.
\end{proof}

These yield the following result, which is clause (2) of Theorem B from the introduction.

\begin{theorem} \label{thm: omega_simul}
  Relative to the consistency of $\ZFC$, it is consistent that $\mathfrak{b} = \mathfrak{d} 
  = \omega_{\omega+2}$ and $\bigwedge_{2 \leq k < \omega} \lim^k \mathbf{A} \neq 0$.
\end{theorem}

\begin{proof}
  This follows immediately from Lemmas \ref{lemma:long-chain-consistency} and 
  \ref{lemma:long-chain-implication}, together with the arguments from the proof of 
  Theorem \ref{thm: n_simul}.
\end{proof}

Notice that Theorems \ref{thm: n_simul} and \ref{thm: omega_simul} conspicuously do not 
mention $\operatorname{lim}^1 \mb{A}$. 
By an argument similar to one due to Kamo \cite{kamo}, if $\lambda \geq \aleph_2$ is a regular 
cardinal and one adds $\lambda$-many Hechler reals to a model of $\ZFC$, then the forcing 
extension satisfies $\operatorname{lim}^1 \mb{A} = 0$. (A special case of this fact was noted in 
\cite{svhdl}.) For completeness, we provide a proof of this fact.

\begin{theorem}
  Suppose that $\lambda \geq \aleph_2$ is a regular cardinal, and let $\bb{H}_\lambda$ 
  be the finite-support iteration of Hechler forcings of length $\lambda$. Then, in 
  $V^{\bb{H}_\lambda}$, we have $\operatorname{lim}^1 \mb{A} = 0$.
\end{theorem}

\begin{proof}
  We think of conditions in $\bb{H}_\lambda$ as being functions whose domains are finite 
  subsets of $\lambda$. For all $p \in \bb{H}_\lambda$ and all $\alpha \in \dom(p)$, 
  $p(\alpha)$ is of the form $(\dot{s}^{p(\alpha)}, \dot{f}^{p(\alpha)})$. 
  By passing to a dense subset of $\bb{H}_\lambda$, we can assume that we only work 
  with conditions $p \in \bb{H}_\lambda$ such that, for all $\alpha \in \dom(p)$, there 
  is $s^{p(\alpha)} \in {^{<\omega}}\omega$ such that $p \restriction \alpha 
  \Vdash_{\bb{H}_\alpha} \dot{s}^{p(\alpha)} = s^{p(\alpha)}$. We will thus write, e.g., 
  $p(\alpha) = (s^{p(\alpha)}, \dot{f}^{p(\alpha)})$, where $s^{p(\alpha)} \in 
  {^{<\omega}}\omega$. For each $\alpha < \lambda$, let $\dot{h}_\alpha$ be the 
  canonical $\bb{H}_\lambda$-name for the Hechler real added by the $\alpha^{\mathrm{th}}$ 
  iterand of $\bb{H}_\lambda$. For all $\gamma < \lambda$, let $\bb{H}_\gamma$ be the 
  initial segment of $\bb{H}_\lambda$ of length $\gamma$.
  
  Suppose for the sake of contradiction that there is $p \in \bb{H}_\lambda$ and an 
  $\bb{H}_\lambda$-name $\dot{\Phi} = \langle \dot{\varphi}_{\dot{f}} \mid \dot{f} \in 
  ({^{\omega}}\omega)^{V^{\bb{H}_\lambda}} \rangle$ such that $p$ forces $\dot{\Phi}$ to 
  be a nontrivial, coherent, $\bb{Z}$-valued $1$-family. For ease of notation, assume that 
  $p = \emptyset$. Assume that each $\dot{\varphi}_{\dot{f}}$ is a nice name; for instance,
  we can take it to be a map from $(\omega \times \omega) \times \bb{Z}$ to the set of 
  antichains of $\bb{H}$.
  
  Note that, for all $\gamma \in S^\lambda_{>\omega}$, $\langle \dot{h}_\alpha \mid 
  \alpha < \gamma \rangle$ is forced to be $<^*$-increasing and cofinal in 
  $({^{\omega}}\omega)^{V^{\bb{H}_\gamma}}$. Moreover, by the chain condition of 
  $\bb{H}_\lambda$, there is a club $D \subseteq \lambda$ such that, for all 
  $\gamma \in D$ and all $\alpha < \gamma$, $\dot{\varphi}_{\dot{h}_\alpha}$ is 
  an $\bb{H}_\gamma$-name. 
  
  Let $G$ be $\bb{H}_\lambda$-generic over $V$. For all $\gamma < \lambda$, let 
  $G_{\gamma}$ be the $\bb{H}_\gamma$-generic filter induced by $G$. We denote 
  the interpretations of names in the forcing extension by removing dots, e.g., 
  $\varphi_{h_\alpha} = (\dot{\varphi}_{\dot{h}_\alpha})_G$. Note that, for all 
  $\gamma \in D$, we have $\langle \varphi_{h_\alpha} \mid \alpha < \gamma \rangle 
  \in V[G_\gamma]$.
  
  \begin{claim}
    There is $\gamma \in D \cap S^\lambda_{> \omega}$ such that 
    $\langle \varphi_{h_\alpha} \mid \alpha < \gamma \rangle$ is nontrivial 
    in $V[G_\gamma]$.
  \end{claim}
  
  \begin{proof}
    Suppose not. For each $\gamma \in D \cap S^\lambda_{> \omega}$, fix 
    $\psi^\gamma \colon \omega \times \omega \ra \bb{Z}$ in $V[G_\gamma]$ 
    trivializing $\langle \varphi_{h_\alpha} \mid \alpha < \gamma \rangle$. 
    By the chain condition of $\bb{H}_\gamma$, there is in fact 
    $\beta_\gamma < \gamma$ such that $\psi^\gamma \in V[G_{\beta_\gamma}]$. 
    Fix a stationary $S \subseteq D \cap S^\lambda_{>\omega}$ and a $\beta < \lambda$ 
    such that $\beta_\gamma = \beta$ for all $\gamma \in S$.
    
    \begin{subclaim}
      For all $\gamma < \gamma'$, both in $S$, there is $k < \omega$ such that 
      $\psi^\gamma \restriction ([k,\omega) \times \omega) = \psi^{\gamma'} \restriction 
      ([k,\omega) \times \omega)$.
    \end{subclaim}
    
    \begin{proof}
      Suppose for the sake of contradiction that $\gamma < \gamma'$ formed a counterexample. 
      Then, since $\psi^\gamma$ and $\psi^{\gamma'}$ are both in 
      $V[G_\beta]$, we can find $f \in ({^{\omega}}\omega)^{V[G_\beta]}$ such that 
      $\psi^\gamma \restriction I_f \neq^* \psi^{\gamma'} \restriction I_f$. 
      Since $h_\beta$ is Hechler-generic over $V[G_\beta]$, we have $f <^* h_\beta$, 
      and hence $\psi^\gamma \restriction I_{h_\beta} \neq^* \psi^{\gamma'} \restriction 
      I_{h_\beta}$. But this contradicts the fact that $\beta < \gamma < \gamma'$, and hence 
      $\psi^\gamma \restriction I_{h_\beta} =^* \varphi_{h_\beta} =^* \psi^{\gamma'} \restriction 
      I_{h_\beta}$.
    \end{proof}
    
    Fix an arbitrary $\gamma \in S$, and an arbitrary $\alpha < \lambda$. Let $\gamma' = 
    \min(S \setminus (\alpha + 1))$. Then $\psi^{\gamma'} \restriction I_{h_\alpha} =^* 
    \varphi_{h_\alpha}$ and, by the above subclaim, we have $\psi^\gamma \restriction I_{h_\alpha} 
    = \psi^{\gamma'} \restriction I_{h_\alpha}$. It follows that $\psi^\gamma$ trivializes 
    $\langle \varphi_{h_\alpha} \mid \alpha < \lambda \rangle$ in $V[G]$, contradicting the 
    fact that $\langle h_\alpha \mid \alpha < \lambda \rangle$ is $<^*$-cofinal in 
    $({^{\omega}}\omega)^{V[G]}$ and $\dot{\Phi}$ was forced to be nontrivial.
  \end{proof}
  
  Now fix $\gamma \in D \cap S^\lambda_{>\omega}$ as given in the claim, and work in 
  $V[G_\gamma]$. Let $\bb{H}_{[\gamma,\lambda)}$ denote the quotient forcing 
  $\bb{H}_\lambda / G_\gamma$. We can think of conditions in $\bb{H}_{[\gamma,\lambda)}$ as 
  being functions whose domains are finite subsets of the interval $[\gamma, \lambda)$ 
  and, for each $p \in \bb{H}_{[\gamma,\lambda)}$ and each $\delta \in \dom(p)$, 
  $p(\delta)$ is of the form $(s^{p(\delta)}, \dot{f}^{p(\delta)})$. 
  
  For each $\alpha < \gamma$, fix a condition $p_\alpha \in \bb{H}_{[\gamma,\lambda)}$ 
  and a $k_\alpha < \omega$ such that $p_\alpha$ forces the following:
  \begin{itemize}
    \item $h_\alpha <_{k_{\alpha}} \dot{h}_\gamma$ (i.e., $h_\alpha(j) < h_\gamma(j)$ for all 
    $j \in [k_\alpha,\omega)$;
    \item $\varphi_{h_\alpha} =_{k_\alpha} \varphi_{h_\gamma}$, i.e., for all 
    $(j,m) \in I_{h_\alpha}$ with $j \geq k_\alpha$, we have $\varphi_{h_\alpha}(j,m) = 
    \varphi_{h_\gamma}(j,m)$.
  \end{itemize}
  Now, using the fact that $\cf(\alpha) > \omega$, and hence $\bb{H}_{[\gamma,\lambda)}$ is 
  $\cf(\alpha)$-Knaster, we may fix an unbounded $A \subseteq \gamma$ and a $k < \omega$ such that
  \begin{itemize}
    \item the conditions $\{p_\alpha \mid \alpha \in A\}$ are pairwise compatible; and
    \item for all $\alpha \in A$, we have $k_\alpha = k$.
  \end{itemize}
  
  \begin{claim}
    For all $\alpha < \alpha'$, both in $A$, we have $\varphi_{h_\alpha} =_k 
    \varphi_{h_{\alpha'}}$.
  \end{claim}
  
  \begin{proof}
    Fix $(j,m) \in I_{h_\alpha} \cap I_{h_{\alpha'}}$ with $j \geq k$.
    Let $p$ be a common extension of $p_\alpha$ and $p_{\alpha'}$. Then $p$ forces:
    \begin{itemize}
      \item $h_\alpha, h_{\alpha'} <_k \dot{h}_\gamma$;
      \item $\dot{\varphi}_{\dot{h}_\gamma} =_k \varphi_{h_\alpha}$ and 
      $\dot{\varphi}_{\dot{h}_\gamma} =_k \varphi_{h_{\alpha'}}$.
    \end{itemize}
    In particular, $p$ forces that $\varphi_{h_\alpha}(j,m) = \dot{\varphi}_{\dot{h}_\gamma}(j,m) = 
    \varphi_{h_{\alpha'}}(j,m)$. It follows that $\varphi_{h_\alpha} =_k \varphi_{h_{\alpha'}}$, 
    as desired.
  \end{proof}
  Let $\psi_0 = \bigcup \{\varphi_{h_\alpha} \restriction (I_{h_\alpha} \cap 
  ([k,\omega) \times \omega)) \mid \alpha \in A \}$. By the previous claim, $\psi_0$ is a 
  function, and its domain includes $I_{h_\alpha}$ mod finite for all $\alpha < \gamma$. 
  Arbitrarily extend $\psi_0$ to a function $\psi$ defined on all of $\omega \times \omega$. 
  By construction, we have $\psi \restriction I_{h_\alpha} =^* \varphi_{h_\alpha}$ for all 
  $\alpha \in A$; since $A$ is cofinal in $\gamma$, it follows that $\psi$ trivializes 
  $\langle \varphi_{h_\alpha} \mid \alpha < \gamma \rangle$, contradicting our choice of 
  $\gamma$ and finishing the proof of the theorem..
\end{proof}

By this theorem, it follows that the 
models we construct to witness Theorems \ref{thm: n_simul} and \ref{thm: omega_simul} will 
satisfy $\lim^1 \mb{A} = 0$. The following question therefore remains open.

\begin{question}
  Is $\bigwedge_{1 \leq k < \omega} \lim^k \mb{A} \neq 0$ consistent with $\ZFC$?
\end{question}

\bibliographystyle{plain}
\bibliography{bib}

\end{document}